\documentclass[11pt]{article} 

\usepackage{fourier} 
\usepackage{amsmath,amssymb,amsthm}
\usepackage{float}
\usepackage{fullpage}
\usepackage{subcaption} 
\usepackage{graphicx} 
\usepackage{authblk}

\usepackage{url} 
\usepackage{color}

\usepackage[boxruled]{algorithm2e} 
\usepackage[colorlinks=true,citecolor=red,linkcolor=blue]{hyperref} 

\newtheorem{theorem}{Theorem}
\newtheorem{lemma}{Lemma}

\newtheorem{proposition}{Proposition}

\theoremstyle{remark}
\newtheorem{remark}{Remark}

\theoremstyle{definition}
\newtheorem{definition}{Definition}
 
\renewcommand{\P}{\mathcal{P}} 
\newcommand{\C}{\mathcal{C}} 
\newcommand{\mbb}[1]{\mathbb{#1}} 
 
\DeclareMathOperator*{\argmin}{arg\,min}

\DeclareMathOperator{\sgn}{sign}
\begin{document} 
\title{Weighted Message Passing and Minimum Energy Flow for Heterogeneous Stochastic Block Models with Side Information}

\date{}
\author[1]{T. Tony Cai}
\author[2]{Tengyuan Liang}
\author[1]{Alexander Rakhlin}

\affil[1]{Department of Statistics, The Wharton School at the University of Pennsylvania}
\affil[2]{Econometrics and Statistics, The University of Chicago Booth School of Business} 

\renewcommand\Authands{ and }

\maketitle 
\begin{abstract}
	We study the misclassification error for community detection in general heterogeneous stochastic block models (SBM) with noisy or partial label information. We establish a connection between the misclassification rate and the notion of minimum energy on the local neighborhood of the SBM. We develop an optimally weighted message passing algorithm to reconstruct labels for SBM based on the minimum energy flow and the eigenvectors of a certain Markov transition matrix.
	The general SBM considered in this paper allows for unequal-size communities, degree heterogeneity, and different connection probabilities among blocks. We focus on how to optimally weigh the message passing to improve misclassification.
\end{abstract}

\section{Introduction}
%
The stochastic block model (SBM), or planted partition model, is a celebrated model that captures the clustering or community structure in large networks. Fundamental phase transition phenomena and limitations for efficient algorithms have been established for the  ``vanilla'' SBM, with equal-size communities \cite{coja2010graph,decelle2011asymptotic,massoulie2014community, mossel2012stochastic,mossel2013belief,krzakala2013spectral,abbe2014exact,hajek2014achieving,abbe2015community, deshpande2015asymptotic}. However, when applying the algorithms to real network datasets, one needs to carefully examine the validity of the vanilla SBM model. First, real networks are heterogeneous and imbalanced; they are often characterized by unequal community size, degree heterogeneity, and distinct connectivity strengths across communities. Second, in real networks, additional side information is often available. This additional information may come, for instance, in the form of a small portion of revealed community memberships, or in the form of node features, or both. In this paper, we aim to address the above concerns by answering the following questions:

\medskip
\noindent \textit{Algorithm}~~ For a general stochastic block model that allows for heterogeneity and contains noisy or partial side information, how to utilize this information to achieve better classification performance? 

\medskip
\noindent \textit{Theory}~~ What is the transition boundary on the signal-to-noise ratio for a general heterogeneous stochastic block model? Is there a physical explanation for the optimal misclassification error one can achieve?

\subsection{Problem Formulation}
\label{sec:problem.formulation}
We define the general SBM with parameter bundle $(n, k, N \in \mathbb{R}^k, Q \in \mathbb{R}^{k \times k})$ as follows. Let $n$ denote the number of nodes and $k$ the number of communities. The vector $N=[n_1, n_2,\ldots, n_k]^T$ denotes the number of nodes in each community. The symmetric matrix $Q = [Q_{ij}]$ represents the connection probability: $Q_{ij}$ is the probability of a connection between a node in community $i$ to a node in community $j$. Specifically, one observes a graph $G(V,E)$ with $|V| = n$, generated from SBM as follows. There is a latent disjoint partition that divides $V=\bigcup_{l=1}^k V_l$ into $k$ communities. Define $\ell(\cdot): V \rightarrow [k]$ to be the label (or, community) of a node $v$. For any two nodes $v, u \in V$, there is an edge between $(u\leftrightarrow v) \in E$ with probability $Q_{\ell(u), \ell(v)}$. The goal is to recover the latent label $\ell(v)$ for each node $v$.
Here we consider the following kinds of heterogeneity: unequal size communities (represented by $[n_i]$), different connection probabilities across communities (as given by $[Q_{ij}]$), and degree heterogeneity (due to both $[n_i]$ and $[Q_{ij}]$).

We study the problem when either noisy or partial label information is available in addition to the graph structure and show how to ``optimally'' improve the classification result (in terms of misclassification error). We argue that this is common for many practical problems. First, in real network datasets, a small portion of labels (or, community memberships) is often available. Second, a practitioner often has certain initial guess of the membership, either through training regression models using node features and partially revealed labels as side information, or running certain clustering algorithms (for example, spectral clustering using non-backtracking matrix, semi-definite programs or modularity method) on a subset or the whole network. We will show that as long as these initial guesses are better than random assignments, one can ``optimally weigh'' the initial guess according to the network structure to achieve small misclassification error.  

Formally, the noisy (or partial) information is defined as a labeling $\tilde{\ell}_{\rm prior}$ on the nodes of the graph with the following stochastic description. The parameter $\delta$ quantifies either (a) the portion of randomly revealed true labels (with the rest of entries in $\tilde{\ell}_{\rm prior}$ missing), or (b) the accuracy of noisy labeling $\tilde{\ell}_{\rm prior}$, meaning 
$$\mathbb{P} (\tilde{\ell}_{\rm prior}(v) = \ell(v)) =  \frac{1-\delta}{k} + \delta,$$ 
and when $\tilde{\ell}_{\rm prior}(v) \neq \ell(v)$, each label occurs with equal probability.


\subsection{Prior Work}
In the literature on vanilla SBM (equal size communities, symmetric case), there are two major criteria --- weak and strong consistency. Weak consistency asks for recovery better than random guessing in a sparse random graph regime ($p, q \asymp 1/n$), and strong consistency requires exact recovery for each node above the connectedness theshold ($p, q \asymp \log n/n $). Interesting phase transition phenomena in weak consistency for SBM have been discovered in \cite{decelle2011asymptotic} via the insightful cavity method from statistical physics.  
Sharp phase transitions for weak consistency have  been thoroughly investigated in \cite{coja2010graph,mossel2012stochastic,mossel2013belief,mossel2013proof,massoulie2014community}. In particular for $k=2$, spectral algorithms on the non-backtracking matrix have been studied in \cite{massoulie2014community} and the non-backtracking walk in \cite{mossel2013proof}. In these two fundamental papers, the authors resolved the conjecture on the transition boundary for weak consistency posed in \cite{decelle2011asymptotic}. 
Spectral algorithms as initialization and belief propagation as further refinement to achieve better recovery was established in \cite{mossel2013belief}. 
Recent work of \cite{abbe2015detection} establishes the positive detectability result down to the Kesten-Stigum bound for all $k$ via a detailed analysis of a modified version of belief propagation.
For strong consistency, \cite{abbe2014exact,hajek2014achieving,hajek2015achieving} established the phase transition using information-theoretic tools and semi-definite programming (SDP) techniques. In the statistics literature, \cite{zhang2015minimax,gao2015achieving} investigated the misclassification rate of the standard SBM.

For the general SBM with connectivity matrix $Q$, \cite{guedon2016community, cai2015robust, chen2015convexified} provided sharp non-asymptotic upper bound analysis on the performance of a certain semi-definite program. They investigated the conditions on $Q$ for a targeted recovery accuracy, quantified as the loss (as a matrix norm) between the SDP solution and the ground truth. The results are more practical for heterogeneous real networks. However, for the analysis of SDP to work, these results all assume certain density gap conditions, i.e., $\max_{1\leq i <  j \leq k} Q_{ij} < \min_{1
\leq i \leq r} Q_{ii}$, which could be restrictive in real settings. Our technical approach is different, and does not require the density gap conditions. Moreover, we can quantify more detailed recovery guarantees, for example, when one can distinguish communities $i, j$ from $l$, but not able to tell $i, j$ apart. In addition, our approach can be implemented in a decentralized fashion, while SDP approaches typically do not scale well for large networks.


For SBM with side information, \cite{kanade2014global, cai2016inference, saade2016fast} considered SBM in the semi-supervised setting, where the side information comes as partial labels. \cite{kanade2014global} considered the setting when the labels for a vanishing fraction of the nodes are revealed, and showed that pushing below the Kesten-Stigum bound \cite{kesten1966limit,kesten1966additional} is possible in this setting, drawing a connection to a similar phenomenon in $k$-label broadcasting processes \cite{mossel2001reconstruction}. In addition, \cite{cai2016inference, saade2016fast} studied linearized belief propagation and misclassification error on the partially labeled SBM.

The focus of this paper is on local algorithms, which are naturally suited for distributed computing \cite{linial1992locality} and provide efficient solutions to certain computationally hard combinatorial optimization problems on graphs. For some of these problems, they are good approximations to global algorithms \cite{kleinberg2000small, gamarnik2014limits, parnas2007approximating, nguyen2008constant}. The fundamental limits of local algorithms have been investigated, in particular, in \cite{montanari2015finding} in the context of a sparse planted clique model. We also want to point out that \cite{mossel2016local} studied the local belief propagation and characterized the expected fraction of correctly labeled vertices using fixed point analysis of the density evolution, in the case of vanilla SBM with side information. 

Finally, we briefly review broadcasting processes on trees. Consider a Markov chain on an infinite tree rooted at $\rho$ with branching number $b$. Given the label of the root $\ell(\rho)$, each vertex chooses its label by applying the Markov rule $M$ to its parent's label, recursively and independently. The process is called broadcasting process on trees. One is interested in reconstructing the root label $\ell(\rho)$ given all the $n$-th level leaf labels.
Sharp reconstruction thresholds for the broadcasting process on general trees for the symmetric Ising model setting (each node's label is $\{+,-\}$) have been studied in \cite{evans2000broadcasting}. \cite{mossel2003information} studied a general Markov channel on trees that subsumes $k$-state Potts model and symmetric Ising model as special cases, and established non-census-solvability below the Kesten-Stigum bound. \cite{janson2004robust} extended the sharp threshold to robust reconstruction, where the vertex' labels are contaminated with noise.
The transition thresholds proved in the above literature correspond to the Kesten-Stigum bound $b |\lambda_2(M)|^2  = 1$ \cite{kesten1966limit,kesten1966additional}. 

\subsection{Our Contributions}

The main results of the present paper are summarized as follows.

\medskip
\noindent \textit{Weighted Message Passing}~~ We propose a new local algorithm -- Weighted Message Passing (WMP) -- that can be viewed as linearized belief propagation with a novel weighted initialization. The \textit{optimal weights} are jointly determined by the \textit{minimum energy flow} that captures the imbalance of local tree-like neighborhood of SBM, and by the \textit{second eigenvectors} of the Markov transition matrix for the label broadcasting process. 
As we will show, these initializations are crucial for the analysis of general SBM that is heterogeneous and asymmetric. 

For the technical contribution, we provide non-asymptotic analysis on the evolution of WMP messages. For general number of communities, it is challenging to track the densities of WMP messages during evolution. We overcome the difficulty through introducing carefully chosen weights and then prove concentration-of-measure phenomenon on messages.

\medskip
\noindent \textit{Misclassification Error}~~ We establish a close connection between the \textit{misclassification error} and a notion called \textit{minimum energy} through the optimally weighted message passing algorithm. In fact, we show that asymptotically almost surely, the misclassification error of WMP ${\sf Err}(\hat{\ell}_{\rm wmp})$ satisfies 
\begin{align*}
	{\sf Err}(\hat{\ell}_{\rm wmp}) \leq \exp\left( -\frac{1}{2 \mathbf{E}^* (\theta^{-2})} \right),
\end{align*} 
where $\mathbf{E}^* (\theta^{-2})$ is defined as the \textit{minimum energy} based on the local tree-like neighborhood, with $\theta^2$ chosen as the conductance level on the edges of the tree. Intuitively, the smaller the energy is, the better the misclassification error one can achieve. This result provides a physical interpretation for the misclassification error. In return, the above upper bound provides a principled way of choosing the optimal weights as to minimize the energy determined by the Thomson's principal \cite{lyons2005probability}. This approach is key to dealing with asymmetric and imbalanced local neighborhoods. 

\medskip
\noindent \textit{Transition Boundary}~~ We show that the Kesten-Stigum bound is the sharp boundary for local algorithms on the signal-to-noise ratio for the general heterogeneous SBM. Define the following quantities 	
\begin{align}
		& K := \left[ {\rm diag}(QN) \right]^{-1}  Q {\rm diag}(N),~ M := Q {\rm diag}(N) \label{eqn:K} \\
		& \theta := \lambda_{2} (K), ~\lambda := \lambda_{1}(M),\nonumber\\
		& \text{and} \quad {\rm SNR} := \lambda \theta^2,
\end{align}
where $N, Q$ are defined in Section~\ref{sec:problem.formulation}, and $\lambda_i(\cdot)$ denotes the $i$-th eigenvalue.
Then the Kesten-Stigum bound ${\rm SNR} = 1$ is the threshold for local algorithms. Above it, the minimum energy $\mathbb{E}^*(\theta^{-2})$ is finite, which asserts a valid upper bound on the misclassification error. Below it, the minimum energy diverges and WMP fails. In fact, we show that below the threshold, no local algorithm can perform significantly better than random guessing.

\medskip
\noindent \textit{Set Identification}~~ When the number of communities $k\geq 3$, we define a notion of \textit{set identification} to describe, for two disjoint sets (of communities) $S, T \subset [k]$, whether one can distinguish $S$ from $T$. This notion subsumes as a special case the classic identification when $S, T$ are singletons. However, it describes more general cases when one cannot distinguish the communities inside $S$ and $T$, but is able to distinguish $S$ and $T$. We provide a mathematical description of this fact using the structure of eigenvectors for the Markov transition matrix $K$ defined in \eqref{eqn:K}. Further, we show that one can weigh the labels in the ``most informative direction'' by initializing WMP according to the second eigenvectors.

\subsection{Organization of the Paper}

The paper is organized as follows. Section~\ref{sec:prelim} reviews the background, definitions, and theoretical tools that will be employed to solve the general SBM. To illustrate the main idea behind the theoretical analysis better, we split the main result into two sections. Section~\ref{sec:k=2} resolves the $k=2$ case, where we emphasize the derivation of  WMP as a linearized belief propagation, and, more importantly, detail the initialization of WMP according to minimum energy flow. Then we establish the connection between misclassification and energy. 
In Section~\ref{sec:g-k}, we focus on the general $k \geq 3$ case, where we incorporate an additional layer of weights on the labels introduced by the eigenvectors of the Markov transition matrix. We then describe the mathematical treatment of set identification. Discussions on the gap between local and global algorithms for growing $k$, and on how WMP utilizes the asymmetry follow in the end. Section \ref{sec:num.study} considers the numerical performance of the proposed algorithm. The proofs of the main results are given in Section \ref{sec:proof}.
%

\section{Preliminaries}
\label{sec:prelim}

\subsection{Tree, Branching Number, Flow and Energy}
Let $T_{t}(o)$ denote the tree up to depth $t$ with root $o$. For a node $v$, the set of children is denoted by $\C(v)$, children at depth $d$ denoted by $\C^d(v)$, and the parent of $v$ is denoted by $\P(v)$. We use $|v|$ to denote the depth of $v$ relative to $o$. If we view a tree as an electrical network, one can define the current \textit{flow} and \textit{energy} on the tree \cite{lyons2005probability}. Later in the paper we will show the close connection between these notions and the misclassification error.
\begin{definition}[Electric Flow]
	\label{def:unit_flow}
	A unit flow $\mathbf{i}(\cdot): V\rightarrow \mathbb{R}$ on a tree $T$ is called a valid \textit{unit flow} if $\mathbf{i}(\rightsquigarrow o) = 1$ and for any $v$
	\begin{align*}
		\mathbf{i}(\rightsquigarrow v) = \sum_{u \in \C(v)} \mathbf{i}(\rightsquigarrow u).
	\end{align*}
\end{definition}
\begin{definition}[Energy and Resistance]
	The \textit{energy} $\mathbf{E}(\mathbf{r}, \mathbf{i})$ of a unit flow $\mathbf{i}$ at \textit{resistance level} $\mathbf{r}>0$ is defined as
	\begin{align*}
		\mathbf{E}(\mathbf{r}, \mathbf{i}) := \sum_{v \in T} \mathbf{i}(\rightsquigarrow v)^2 \mathbf{r}^{|v|}.
	\end{align*}
	The \textit{minimum energy} $\mathbf{E}^*(\mathbf{r})$ is
	\begin{align*}
		\mathbf{E}^*(\mathbf{r}) := \inf_{\mathbf{i}} \mathbf{E}(\mathbf{r}, \mathbf{i}),
	\end{align*}
	where the infimum is over all valid unit flows. Denote the minimum energy flow as $\mathbf{i}^*$.
\end{definition}
When assigning resistance $\mathbf{r}^d$ to edges that are $d$-depth away from the root, the energy enjoys the natural physical interpretation. We also remark that for a given resistance level, one can calculate the minimum energy flow $\mathbf{i}^*$ on the tree using Thomson's principal. We identify the reciprocal of resistance level with the conductance level.

Now we are ready to define the \textit{branching number} of a tree $T$ through \textit{minimum energy}.
\begin{definition}[Branching Number]
	\label{def:branch.num}
	The \textit{branching number} ${\rm br}(T)$ can be defined as
	\begin{align*}
		{\rm br}(T) := \sup \{ \mathbf{r}:~ \mathbf{E}(\mathbf{r}) < \infty \} = \sup \{ \mathbf{r}:~ \inf_{\mathbf{i}} \sum_{v \in T} \mathbf{i}(\rightsquigarrow v)^2 \mathbf{r}^{|v|} < \infty \}.
	\end{align*}
\end{definition}
It is well known that the branching number not only captures the growth rate of the tree, but also the more detailed structure, such as imbalance \cite{lyons2005probability}.

\subsection{Broadcasting Trees and SBM}
When viewed locally, stochastic block models in the sparse regime share similarities with a label broadcasting process on a Galton-Watson tree. In fact, the local neighborhood of SBM can be coupled with a broadcasting tree with high probability as $n \rightarrow \infty$. This phenomenon has been investigated in studying the detectability and reconstruction threshold for vanilla SBM (equal-size communities, symmetric case), as in \cite{mossel2012stochastic}.

Let us formally define the \textit{label broadcasting process} conditioned on a tree $T(o)$.
\begin{definition}[Label Broadcasting]
	\label{def:label-broadcasting}
	Given a tree $T(o)$, the $k$-broadcasting process on $T$ with the Markov transition matrix $K \in \mathbb{R}^{k \times k}$ describes the following process of label evolution. Conditioning on a node $v$ and its label $\ell(v) \in [k]$, the labels of children $u \in \C(v)$ are sampled independently from
	\begin{align*}
		\mathbb{P} (\ell(u) | \ell_{T_{|v|}(o)}) = \mathbb{P}(\ell(u) | \ell(v)) = K_{\ell(v),\ell(u)},
	\end{align*}
	where the first equality is the Markov property.
\end{definition}

Let us review the definition of the multi-type Galton-Watson tree. We shall only consider the Poisson branching process.
\begin{definition}[Multi-type Galton-Watson Tree]
	Consider a $k$-types Galton-Watson process with the mean matrix $M \in \mathbb{R}^{k \times k}$. For a node $v$, given its type $\ell(v) = i$, the number of type $j$ children of $v$ enjoys a ${\sf Poisson}(m_{ij})$ distribution, independently of other types. Start the process recursively for $t$ generations from root $o$. The tree $T_t(o)$ is called a multi-type Galton-Watson tree.
\end{definition}



\subsection{Notation}

The moment generating function (MGF) for a random variable $X$ is denoted by  $\Psi_X(\lambda) = \mbb{E} e^{\lambda X}$. 
For asymptotic order of magnitude, we use $a(n) = \mathcal{O}(b(n))$ to denote that $\forall n, a(n) \leq C b(n)$ for some universal constant $C$, and use $\mathcal{O}^*(\cdot)$ to omit the poly-logarithmic dependence. As for notation $\precsim, \succsim$: $a(n) \precsim b(n)$ if and only if $\varlimsup \limits_{n\rightarrow \infty} \frac{a(n)}{b(n)} \leq C$, with some constant $C>0$, and vice versa. The square bracket $[\cdot]$ is used to represent the index set $[k]:=[1,2,\ldots,k]$; in particular when $k=2$, $[2]:=\{+,-\}$ for convenience.

Recall that the hyperbolic tangent is $\tanh x = \frac{e^{x} - e^{-x}}{e^{x} + e^{-x}}$. The message-passing algorithm in the following sections involves a non-linear update rule defined through a function
\begin{align}
	\label{eq:f_theta}
	f_{\theta_1, \theta_2}(x)& := \log \frac{1+\theta_1 \tanh \frac{x}{2}}{1- \theta_2 \tanh \frac{x}{2}},
\end{align}
for $0<\theta_1, \theta_2 <1$. Note that the derivative $f'_{\theta_1,\theta_2}(0) = \frac{\theta_1 + \theta_2}{2}$.

\section{Two Communities}
\label{sec:k=2}

In this Section we will illustrate the main results for the case of two, possibly imbalanced, communities. We motivate the weighted message passing algorithm, and its relation to minimum energy flow. We investigate the connection between misclassification and minimum energy, as well as the corresponding transition threshold for general SBM.

\subsection{Main Algorithmic and Theoretical Results}
This section serves as an informal summary of the results for $k=2$. As a start, we introduce the following weighted message passing (WMP) Algorithm~\ref{alg:wmp}.

\begin{algorithm}[H]
	\KwData{Graph $G(V,E)$ with noisy label information $\tilde{\ell}_{\rm prior}$. Parameters: neighborhood radius $\bar{t}$ and conductance level $\bar{\theta}^2$. }
	\KwResult{The labeling for each node $o \in V$.}
	\For{each node $o \in V$,}{
	Open the tree neighborhood $T_{\bar{t}}(o)$ induced by the graph $G(V,E)$ \;
	Layer $\bar{t}$: for every node $u \in \C^{\bar{t}}(o)$ with distance $\bar{t}$ to the root on $T_{\bar{t}}(o)$, initialize its message $$M(u, 0) =  \bar{\theta}^{-2|u|} \cdot \mathbf{i}^*(\rightsquigarrow u) \cdot \sgn[ \tilde{\ell}_{\rm prior}(u) ],$$ where $\mathbf{i}^*(\rightsquigarrow u)$ is the minimum energy flow to $u$ calculated via Thomson's principal on $T_{\bar{t}}(o)$ with conductance level $\bar{\theta}^2$  \;
	\For{$t = 1,\ldots \bar{t}$,}{
	Layer $\bar{t} - t$: for every node $u \in \C^{\bar{t} - t}(o)$, calculate the message $M(u, t)$ through the linearized update rule
	$$M(u, t) = \sum_{v \in \C(u)} \bar{\theta} M(v, t-1).$$}
	Output $\hat{\ell}_{\rm wmp}(u) = \sgn[ M(o, \bar{t}) ]$.}
	\caption{Weighted Message Passing}
	\label{alg:wmp}
\end{algorithm}
\medskip

We remark that WMP can run in parallel for all nodes due to its decentralized nature. For fixed depth $\bar{t}$ and sparse SBM (when $n \max_{i,j} Q_{ij} \precsim \log n$), the algorithm runs in $\mathcal{O}^*(n)$ time. 

The following theorem is a simplified version of Theorems~\ref{thm:energy-error} and \ref{thm:k=2.lower.bound} below:

\begin{theorem}[General SBM: $k=2$]
	\label{thm:k=2.statement}
	Consider the general stochastic block model $G(V, E)$ with parameter bundle $(n, k=2, N, Q)$, with either partial or noisy label information $\tilde{\ell}_{\rm prior}$ with parameter $0<\delta<1$.
	Assume that $n \max_{i,j} Q_{ij} \precsim n^{o(1)}$.
	For any node $o \in V$ and its depth $t$ leaf labels $\tilde{\ell}_{\rm prior} (\C^t(o))$, define the worst-case misclassification error of a local estimator ~$\sigma_t(o):  \tilde{\ell}_{\rm prior} (\C^t(o)) \rightarrow \{ +, - \}$ as
	\begin{align}
		{\sf Err}(\sigma_t) := \max_{l \in \{+,- \}}~\mbb{P} \left( \sigma_t(o) \neq \ell(o) |  \ell(o) = l  \right).
	\end{align}
	Define
	\begin{align}
		\bar{\theta} &:= \frac{1}{4} \left( \frac{n_1 Q_{11} -  n_2 Q_{12}}{n_1 Q_{11} + n_2 Q_{12}} + \frac{n_2 Q_{22} - n_1 Q_{21}}{n_1 Q_{21} + n_2 Q_{22}}\right)\\
		\lambda &:= \lambda_{1}\left(
		\begin{bmatrix}
		n_1 Q_{11} & n_2 Q_{12}\\
		n_1 Q_{21} & n_2 Q_{22}\\
		\end{bmatrix} \right).
	\end{align}
	Let $\mathbf{E}^*(\bar{\theta}^{-2})$ be the minimum energy on $T_{t}(o)$ with conductance level $\bar{\theta}^2$ as $t\rightarrow \infty$. 
	
	The transition boundary for this general SBM depends on the value $${\sf SNR} = \lambda \bar{\theta}^2.$$
	On the one hand, if $\lambda \bar{\theta}^2 > 1,$
	the WMP Algorithm~\ref{alg:wmp}, denoted as $\hat{\ell}_{\rm wmp}$, enjoys the following upper bound on misclassification
	\begin{align}
	\limsup_{t \rightarrow \infty} \limsup_{n\rightarrow \infty}~ {\sf Err}(\hat{\ell}_{\rm wmp})  \leq  \exp\left(- \frac{1}{2 \mathbf{E}^*(1/\bar{\theta}^2) }  \right) \wedge \frac{1}{2},
	\end{align}
	for any fixed $\delta > 0$. On the other hand, if $\lambda \bar{\theta}^2 < 1,$
	for any local estimator $\sigma_t$ that uses only label information on depth $t$ leaves, the minimax misclassification error is lower bounded by
	\begin{align}
	\liminf_{t \rightarrow \infty} \liminf_{n \rightarrow \infty}~ \inf_{\sigma_t} {\sf Err}(\sigma_t) = \frac{1}{2}.
  \end{align}
\end{theorem}

\begin{remark}
We remark that Algorithm~\ref{alg:wmp} is stated for the case when noisy label information is known for all nodes in layer $\bar{t}$. For the case of partial label information, there are two options to modify the initialization of the algorithm:
(1) view the partial label information with parameter $\delta$ as the noisy label information on layer $\bar{t}$ only, with $\mathbb{P}(\tilde{\ell}_{\rm prior}(u) = \ell(u)) = \delta + (1 - \delta) \frac{1}{2}$ --- with probability $\delta$, the label is revealed exactly, and with probability $1-\delta$, the label is decided using coin-flip --- then proceed with the algorithm;
(2) view the partial information as on each layer there is a $\delta$ portion of nodes whose label is shown exactly. Call the set of these nodes $V^{l}(T_{\bar{t}}(o))$. Then we need to initialize the message $M(u)$ for all $u \in V^{l}(T_{\bar{t}}(o))$ first before using the recursion
$M(u) = \sum_{v \in \C(u)} \bar{\theta} M(v)$. It can be shown that these two treatments enjoy similar asymptotic performance in terms of misclassification error, above the SNR threshold. However, the latter performs better numerically for fixed depth tree as it utilizes more information.
\end{remark}


	We decompose the proof of Theorem~\ref{thm:k=2.statement} into several building steps: (1) conditioned on the local tree structure, prove concentration-of-measure on WMP messages when label propagates according to a Markov transition matrix $K$; (2) for a typical tree instance generated from multi-type Galton-Watson process, establish connection among the misclassification rate, transition boundary and minimum energy through the concentration result; (3) show that in the sparse graph regime of interest, the local neighborhood of general SBM can be coupled with a multi-type Galton-Watson with Markov transition matrix 
	$$K :=
		\begin{bmatrix}
		\frac{n_1 Q_{11}}{n_1 Q_{11} + n_2 Q_{12}} & \frac{n_2 Q_{12}}{n_1 Q_{11} + n_2 Q_{12}}\\
		\frac{n_1 Q_{21}}{n_1 Q_{21} + n_2 Q_{22}} & \frac{n_2 Q_{22}}{n_1 Q_{21} + n_2 Q_{22}}\\
		\end{bmatrix} $$ 
for label broadcasting (the explicit expression based on Eq.~\eqref{eqn:K}). We remark that (3) follows similar proof strategy as in \cite{mossel2012stochastic}, where the coupling for vanilla SBM has been established. The lower bound follows from Le Cam's testing argument, and the difficulty lies in analyzing the distance between measures recursively on the local tree. 

\begin{remark}
	When the local tree is regular and symmetric and $\lambda \bar{\theta}^2 > 1$, the minimum energy can be evaluated exactly as $$\mathbf{E}^*(\bar{\theta}^{-2}) = \frac{1}{\lambda \bar{\theta}^2 - 1},$$ which implies that misclassification error takes the exponentially decaying form $\exp\left( - \frac{{\sf SNR} - 1}{2} \right)$.
	Hence, the result provides a detailed understanding of the strength of the ${\sf SNR}$ and its effect on misclassification, i.e., the inference guarantee. More concretely, for the vanilla SBM in the regime $p = a/n, q = b/n$, the boundary is 
	$
	{\sf SNR} =   \frac{n(p-q)^2}{2(p+q)} >1,
	$
	which is equivalent to the boundary 
	$$\frac{(a-b)^2}{2(a+b)}> 1$$
	for weak consistency  in  \cite{mossel2013proof, massoulie2014community}.
	In addition, one observes that ${\sf SNR} > 1+ 2\log n$ implies 
	${\sf Err}(\hat{\ell}) < 1/n \rightarrow 0$, which asserts strong consistency. This condition on ${\sf SNR}$ is satisfied, for instance, by taking $p = a\log n/n, q = b\log n/n$ in vanilla SBM and 
	computing the relationship between $a, b$ to ensure
	${\sf SNR} =  \frac{n(p-q)^2}{2(p+q)} > 1+ 2\log n.$
	This relationship is precisely
	$$ \frac{\sqrt{a} - \sqrt{b}}{\sqrt{2}} > \sqrt{1+\frac{1}{2\log n}} \cdot \frac{\sqrt{2(a+b)}}{\sqrt{a}+\sqrt{b}} > 1.$$
	The above agrees with the threshold for strong recovery in \cite{abbe2014exact,hajek2014achieving}.
\end{remark}

\subsection{Weighted Message Passing \& Minimum Energy Flow}
In this section, we will motivate our proposed weighted message passing (WMP) from the well-known belief propagation (BP) on trees. There are two interesting components in the WMP Algorithm~\ref{alg:wmp}: the linearization part, and the initialization part. We will discuss each one in details in this section.

Recall the Definition~\ref{def:label-broadcasting} of the label broadcasting process on tree $T(o)$ with $k=2$. For convenience, let us denote the Markov transition matrix $K$ to be
\begin{align}
	\label{eq:K.k=2}
	K =
	\begin{bmatrix}
	\frac{1+\theta_1}{2} & \frac{1-\theta_1}{2} \\
	\frac{1-\theta_2}{2} & \frac{1+\theta_2}{2}
	\end{bmatrix}.
\end{align}
The BP algorithm is the Bayes optimal algorithm on trees given the labels of leaves.
Define for a node $u \in V$ the BP message as
$$B(u, t) := \log \frac{\mathbb{P}(\ell(u) = +) | \ell_{\rm obs}(T_{t}(u)))}{\mathbb{P}(\ell(u) = - | \ell_{\rm obs}(T_{t}(u)))},$$
which is the posterior logit of $u$'s label given the observed labels $\ell_{\rm obs}(T_{t}(u)))$.
Using Bayes rule and conditional independence, one can write out the explicit evolution for BP message through $f_{\theta_1, \theta_2}$ in \eqref{eq:f_theta}
\begin{align}
	B(u,t) &= \sum_{v \in \C(u)} \log \left( \frac{1 + \theta_1 \tanh \frac{B(v,t-1)}{2} }{1 - \theta_2 \tanh \frac{B(v, t-1)}{2}} \right) \nonumber \\
         &= \sum_{v \in \C(u)} f_{\theta_1, \theta_2} \left( B(v, t-1) \right), \label{eq:bp-update}
\end{align}
with $\theta_1,\theta_2$ as in Markov transition matrix $K$.
While the method is Bayes optimal, the density of the messages $B(u,t)$ is difficult to analyze, due to the blended effect of the dependence on revealed labels and the non-linearity of $f_{\theta_1, \theta_2}$. However, the WMP Algorithm~\ref{alg:wmp} --- a linearized BP --- shares the same transition threshold with BP, and is easier to analyze. Above a certain threshold, the WMP succeeds, which implies that the optimal BP will also work. Below the same threshold, even the optimal BP will fail, and so does the WMP. The updating rule for WMP messages $M(u, t)$ is simply a replacement of Eq.~\eqref{eq:bp-update} by its linearized version,
$$M(u,t) = \sum_{v \in \C(u)} \frac{\theta_1+\theta_2}{2} M(v, t-1).$$

The initialization of the WMP messages on the leaves $M(u,0)$ whose labels have been observed is crucial to the control of the misclassification error of the root node, especially for general SBM with \textit{heterogeneous degrees}.
For general SBM, one should expect to initialize the messages according to the detailed local tree structure, where the degree for each node could be very different. It turns out that the optimal misclassification for WMP is related to a notion called the \textit{minimum energy} $\mathbf{E}^*$. Moreover, the optimal initialization for leaf message $u$ is proportional to the \textit{minimum energy flow} $\mathbf{i}^*(\rightsquigarrow u)$ on the local tree, with \textit{conductance level} $\bar{\theta}^2$. In plain language, $\mathbf{i}^*(\rightsquigarrow u)$ provides a quantitative statement of the importance of the vote $u$ has for the root. Note that for imbalanced trees, $\mathbf{i}^*$ could vary significantly from node to node, and can be computed efficiently given the tree structure $T_{t}(o)$ for a specified conductance level.

\subsection{Concentration, Misclassification \& Energy}
We now prove the concentration-of-measure phenomenon on WMP messages. Through the concentration, we will show the close connection between \textit{misclassification} and \textit{energy}. We will first state the result conditioned on the tree structure $T_{t}(o)$.
\begin{lemma}[Concentration on Messages]
	\label{lem:k=2.concentration}
	Recall the label broadcasting process with Markov transition kernel $K \in \mathbb{R}^{2\times 2}$ on tree $T_{\bar{t}}(o)$. Assume the MGF of messages on leaves $M(u, 0)$ satisfies the following
	\begin{align*}
		\mathbb{E}\left[ e^{\lambda M(u,0)} | \ell(u) = + \right] \leq e^{\lambda \mu_0(u, +)} e^{\frac{\lambda^2 \sigma^2_0(u)}{2}} \\
		\mathbb{E}\left[ e^{\lambda M(u,0)} | \ell(u) = - \right] \leq e^{\lambda \mu_0(u, -)} e^{\frac{\lambda^2 \sigma^2_0(u)}{2}}
	\end{align*}
	for any $\lambda$, with parameter
	\begin{align*}
		\mu_0(u) =
		\begin{bmatrix}
		\mu_0(u,+) \\
		\mu_0(u,-)
		\end{bmatrix} \in \mathbb{R}^{2}, \quad \sigma^2_0(u) \in \mathbb{R}.
	\end{align*}
	Define the following updating rules for a node $v$
	\begin{align}
		\mu_t(v) &= \sum_{u \in \C(v)} \bar{\theta} K \mu_{t-1}(u) \\
		\sigma^2_t(v) &= \sum_{u \in \C(v)} \bar{\theta}^2 \left\{ \sigma^2_{t-1}(u) + \left[ \frac{\mu_{t-1}(u, +) - \mu_{t-1} (u, -)}{2}\right]^2  \right\}.
	\end{align}
	Then the following concentration-of-measure holds for the root message $M(o, \bar{t})$:
	\begin{align*}
		M(o, \bar{t})_{| \ell(o)=+} \geq \mu_{\bar{t}}(o, +) - x \cdot \sigma_{\bar{t}}(o) \\
		M(o, \bar{t})_{| \ell(o)=-} \leq \mu_{\bar{t}}(o, -) + x \cdot \sigma_{\bar{t}}(o)
	\end{align*}
	both with probability $1 - \exp(-\frac{x^2}{2})$.

	In addition, if we choose $\frac{\mu_{\bar{t}}(o, +)+\mu_{\bar{t}}(o, -)}{2}$ as the cut-off to provide classification $\hat{\ell}_{\rm wmp}$, then the misclassification error is upper bounded by
	\begin{align}
		\exp \left(- \frac{[\mu_{\bar{t}}(o, +) - \mu_{\bar{t}}(o, -)]^2}{8\sigma^2_{\bar{t}}(o)}\right).
	\end{align}
\end{lemma}
The above Lemma provides an expression on the classification error. The next Theorem will show that with the ``optimal'' initialization for WMP, the misclassification error is connected to the minimum energy.

\begin{theorem}[Connection between Misclassification \& Energy]
	\label{thm:energy-error}
	Define the current flow $$\mathbf{i}(\rightsquigarrow v) =\frac{\bar{\theta}^{2|v|} [\mu_{t - |v|}(v,+) - \mu_{t - |v|}(v,-)] }{[\mu_{t}(o, +) - \mu_{t}(o, -)]}.$$ Then it is a valid unit flow on $T_{t}(o)$, and
	the following equation holds
	\begin{align*}
		\frac{\sigma^2_{t}(o)}{\left[ \frac{[\mu_{t}(o, +) - \mu_{t}(o, -)]}{2} \right]^2} = (1 + o_{t}(1))\sum_{v \in T_{t}(o)} \mathbf{i}(\rightsquigarrow v)^2 \left( \bar{\theta}^{-2} \right)^{|v|} = (1 + o_{t}(1)) \mathbf{E}_{t}(\mathbf{i}, \bar{\theta}^{-2})
	\end{align*}
	when $\lim_{t \rightarrow \infty} \mathbf{E}_t (\mathbf{i}, \bar{\theta}^{-2}) <\infty$.
	Moreover, if we choose $\mu_0(v)$ so that $\mathbf{i}$ is the minimum energy flow, then
	under the condition
	$${\rm br}[T(o)] \bar{\theta}^2 > 1,$$
	we have $\mathbf{E}^*(\bar{\theta}^{-2}) < \infty$ and
	\begin{align}
		\lim_{t \rightarrow \infty} \inf_{\mathbf{i}}~ \frac{\sigma^2_{t}(o)}{\left[ \frac{[\mu_{t}(o, +) - \mu_{t}(o, -)]}{2} \right]^2} \leq  \sum_{v \in T(o)} \mathbf{i}^*(\rightsquigarrow v)^2 \left( \bar{\theta}^{-2} \right)^{|v|} =  \mathbf{E}^*(\bar{\theta}^{-2}).
	\end{align}
\end{theorem}
\begin{remark}
	The above Theorem~\ref{thm:energy-error} and Lemma~\ref{lem:k=2.concentration} together state the fact that if ${\rm br}[T(o)] \bar{\theta}^2 > 1$, $\mathbf{E}^*(\bar{\theta}^{-2})$ is finite, and the optimal initialization of WMP enjoys the  asymptotic misclassification error bound of
	$$\exp\left(-\frac{1}{2\mathbf{E}^*(\bar{\theta}^{-2})}\right).$$
	Qualitatively, the smaller the minimum energy is, the smaller the misclassification error is, and it decays exponentially. On the contrary, if the minimum energy is infinite (${\rm br}[T(o)] \bar{\theta}^2 <1$), the misclassification error bound for WMP becomes vacuous. Another remark is that when the tree is regular, the minimum energy takes the simple form $\mathbf{E}^*(\bar{\theta}^{-2}) = \frac{1}{{\rm br}[T(o)] \bar{\theta}^2 - 1}$, which implies the upper bound $\exp(-\frac{{\rm br}[T(o)] \bar{\theta}^2 - 1}{2})$ on asymptotic misclassification error.
\end{remark}

\subsection{Below the Threshold: Limitation of Local Algorithms}
In this section, we will show that the $\sf SNR$ threshold (for WMP algorithm) is indeed sharp for the local algorithm class. The argument is based on Le Cam's method. Let us prove a generic lower bound for any fixed tree $T_{t}(o)$, and for the $k=2$ label broadcasting process with transition matrix $K$ (as in Eq.~\eqref{eq:K.k=2}).

\begin{theorem}[Limitation of Local Algorithms]
	\label{thm:k=2.lower.bound}
	Recall the label broadcasting process with Markov transition kernel $K$ on tree $T_t(o)$. Consider the case when noisy label information (with parameter $\delta$) is known on the depth-$t$ layer leaf nodes.
	Denote the following two measures $\pi^{+}_{\ell_{T_{t}(o)}},  \pi^{-}_{\ell_{T_{t}(o)}}$ as distributions on leaf labels given $\ell(o)= +, -$ respectively. Under the condition 
	$$
	{\rm br}[T(o)] \bar{\theta}^2<1,$$ if 
	$\log (1 + \frac{4\delta^2}{1 - \delta^2}) \leq 1 - {\rm br}[T(o)] \bar{\theta}^2$, 
	the following equality on total variation holds
	\begin{align*}
	\lim_{t \rightarrow \infty}~ d_{\rm TV}^2 \left( \pi^{+}_{\ell_{T_{t}(o)}},  \pi^{-}_{\ell_{T_{t}(o)}} \right) = 0.
	\end{align*}
	Furthermore, the above equation implies
	$$
	\lim_{t \rightarrow \infty} \inf_{\sigma_t} \sup_{l \in \{+,-\}} ~ \mbb{P} \left( \sigma_t(o) \neq \ell(o) | \ell(o) = l \right) = \frac{1}{2}
	$$
	where~ $\sigma_t(o): \tilde{\ell}_{\rm prior}(\C^t(o)) \rightarrow \{ +, - \}$ is any estimator mapping the prior labels in the local tree to a decision.
\end{theorem}
The above theorem is stated under the case when the noisy label information is known and only known for all nodes in layer $t$. One can interpret the result as, below the threshold ${\rm br}[T(o)] \bar{\theta}^2<1$, one cannot do better than random guess for the root's label based on noisy leaf labels at depth $t$ as $t \rightarrow \infty$. The proof relies on a technical lemma on branching number and cutset as in \cite{pemantle1999robust}. We would like to remark that the condition $\log (1 + \frac{4\delta^2}{1 - \delta^2}) \leq 1 - {\rm br}[T(o)] \bar{\theta}^2$ can be satisfied when $\delta$ is small.

\section{General Number of Communities}
\label{sec:g-k}

In this section, we will extend the algorithmic and theoretical results to the general SBM for any fixed $k$ or growing $k$ with a slow rate (with respect to $n$).
There are several differences between the general $k$ case and the $k=2$ case. First, algorithmically, the procedure for general $k$ requires another layer of weighted aggregation besides the weights introduced by minimum energy flow (according to the detailed tree irregularity). The proposed procedure introduces the weights on the types of labels ($k$ types) revealed, and then aggregates the information in the most ``informative direction'' to distinguish the root's label. Second, the theoretical tools we employ enable us to formally describe the intuition that in some cases for general SBM, one can distinguish the communities $i,j$ from $k$, but not being able to tell $i$ and $j$ apart. We will call this the set identification.

\subsection{Summary of Results}
We summarize in this section the main results for general SBM with $k$ unequal size communities, and introduce the corresponding weighted message passing algorithm (WMP).

We need one additional notation before stating the main result. For a vector $w \in \mathbb{R}^k$, assume there are $m$ unique values for $w_l, l \in [k]$. Denote by $S_i, 1\leq i \leq m,$ the sets of equivalent values associated with $w$ --- for any $l,l' \in [k]$, $w_{l} = w_{l'}$ if and only if $l, l'\in S_i$ for some $i \in [m]$. Denote $w_{S_i}$ to be the equivalent value $w_{l}, l\in S_i$.

\begin{theorem}[General SBM: $k$ communities]
	\label{thm:k.statement}
	Consider the general stochastic block model $G(V, E)$ with parameter bundle $(n, k, N, Q)$, with either partial or noisy label information $\tilde{\ell}_{\rm prior}$ with parameter $0<\delta<1$.
	Assume that $n \max_{i,j} Q_{ij} \precsim n^{o(1)}$.
	For any node $o \in V$ and its depth $t$ leaf labels $\tilde{\ell}_{\rm prior} (\C^t(o))$, define the set misclassification error of a local estimator ~$\sigma_t(o): \tilde{\ell}_{\rm prior}(\C^t(o)) \rightarrow [k]$ as,
	\begin{align}
		{\sf Err}_{S,T}(\sigma_t) := \max \left\{ \mbb{P} \left( \sigma_t(o) \in S |  \ell(o) \in T  \right), \mbb{P} \left( \sigma_t(o) \in T |  \ell(o) \in S  \right) \right\},
	\end{align}
	where $S, T \subset [k]$ are two disjoint subsets.
	Define 
	\begin{align}
		& K := \left[ {\rm diag}(QN) \right]^{-1}  Q {\rm diag}(N), ~ M = Q {\rm diag}(N) \\
		& \theta := \lambda_{2} (K), ~\lambda := \lambda_{1}(M).
	\end{align}
	Let $\mathbf{E}^*(1/\theta^2)$ be the minimum energy on $T_{t}(o)$ with conductance level $\theta^2$ as $t\rightarrow \infty$. Assume that $K$ is symmetric and denote $V \in \mathbb{R}^k$ to be the space spanned by the second eigenvectors of $K$. Choose any $w \in V, w \perp \mathbf{1}$ as the initialization vector in WMP Algorithm~\ref{alg:wmp-k}.

	On the one hand, when $\lambda \theta^2 > 1,$
		the WMP Algorithm~\ref{alg:wmp-k} initialized with $w$ outputs $\hat{\ell}_{\rm wmp}$ that can distinguish the indices set $S_i, 1\leq i \leq m$
		\begin{align}
		\limsup_{t \rightarrow \infty} \limsup_{n\rightarrow \infty} \max_{i,j \in [m]}~ {\sf Err}_{S_i, S_j}(\hat{\ell}_{\rm wmp})  \leq  \exp\left(- \frac{R^2}{2 \mathbf{E}^*(1/\theta^2) }  \right),
		\end{align}
		for any fixed $\delta > 0$, where $R^2 = \frac{\min_{i,j} | w_{S_i} - w_{S_j}|}{\max_{i,j} | w_{S_i} - w_{S_j}|}$.

  	On the other hand, if $\lambda \theta^2 < 1,$
  	for any $t$-local estimator $\sigma_t$ that only based on layer $t$'s noisy labels, the minimax misclassification error is lower bounded by
  	\begin{align}
  	\liminf_{t \rightarrow \infty} \liminf_{n \rightarrow \infty}~ \inf_{\sigma_t} \sup_{i,j \in [k], i \neq j} {\sf Err}_{i, j}(\sigma_t) \geq  \frac{1}{2k}.
    \end{align}
\end{theorem}

	The proof for general $k$ case requires several new ideas compared to the $k=2$ case. 
	Let us first explain the intuition behind some quantities here. Again we focus on the case when the network is sparse, i.e. $n \max_{i,j} Q_{ij} \precsim n^{o(1)}$. According to the coupling Proposition~\ref{lma:coupling.tree}, one can focus on the coupled multi-type Galton-Watson tree, for a shallow local neighborhood of a node $o$. $K \in \mathbb{R}^{k\times k}$ then denotes the transition kernel for the label broadcasting process on the tree, and $\lambda$ denotes the branching number of the multi-type Galton-Watson tree. The transition threshold $\lambda\theta^2 = 1$, also called Kesten-Stigum bound, has been well-studied for reconstruction on trees \cite{kesten1966additional, kesten1966limit,mossel2001reconstruction, janson2004robust}. Our contribution lies in establishing the connection between the set misclassification error, minimum energy flow, as well as the second eigenvectors of $K$. This is done through analyzing Algorithm~\ref{alg:wmp-k} (to be introduced next) with a novel initialization of the messages, using both minimum energy flow and the eigenvectors of $K$. 

\begin{remark}
	One distinct difference between the general $k$ case and the $k=2$ case is the notion of set misclassification error, or set identification. This formalizes the intuition that for general SBM that is asymmetric and imbalanced, it may be possible to distinguish communities $i,j$ from community $l$, yet not possible to tell $i$ and $j$ apart. The above Theorem provides a mathematical description of the phenomenon, for any initialization using vectors in the eigen-space corresponding to the second eigenvalue.
\end{remark}

The key new ingredient compared to the Algorithm~\ref{alg:wmp} is the introduction of additional weights $w \in \mathbb{R}^k$ on the labels. The choice of $w$ will become clear in a moment. 

\begin{algorithm}[H]
	\KwData{Same as in Algorithm~\ref{alg:wmp} and an additional weight vector $w \in \mathbb{R}^k$.}
	\KwResult{The labeling for each node $o \in V$.}
	\For{each node $o \in V$,}{
	Open the tree neighborhood $T_{\bar{t}}(o)$ \;
	Layer $\bar{t}$: for every node $u \in \C^{\bar{t}}(o)$, initialize its message
	$$M(u, 0) =  \theta^{-2|u|} \cdot \mathbf{i}^*(\rightsquigarrow u) \cdot w_{\tilde{\ell}_{\rm prior}(u)},$$
	where $w_{\tilde{\ell}_{\rm prior}(u)}$ denotes the $\tilde{\ell}_{\rm prior}(u)$-th coordinate of the weight vector $w$, $\mathbf{i}^*(\rightsquigarrow u)$ is the minimum energy flow \;
	Initialize parameters $\mu_0(u) \in \mathbb{R}^k, \sigma^2_0(u) \in \mathbb{R}$ as
	\begin{align*}
		\mu_0(u, l) &= \delta \cdot \theta^{-2|u|} \mathbf{i}^*(\rightsquigarrow u) \cdot w_{l}, ~\text{for}~ l \in [k] \\
		\sigma^2_0(u) &= \left( \theta^{-2|u|} \mathbf{i}^*(\rightsquigarrow u) \right)^2 \cdot \max_{i,j \in [k]} |w_i - w_j|^2
	\end{align*}
	\For{$t = 1,\ldots \bar{t}$,}{
	Layer $\bar{t} - t$: for every node $u \in \C^{\bar{t} - t}(o)$, update message $M(u, t)$ through the linearized rule
	$$M(u, t) = \sum_{v \in \C(u)} \theta M(v, t-1).$$
	Update the parameters $\mu_t(u) \in \mathbb{R}^k, \sigma^2_t(u) \in \mathbb{R}$
	{\small
	\begin{align*}
		\mu_t(u) &= \sum_{v \in \C(u)} \theta K \mu_{t-1}(v) \\
		\sigma^2_t(u) &= \sum_{v \in \C(u)} \theta^2 \left\{ \sigma^2_{t-1}(v) + \left[ \frac{ \max\limits_{i,j \in [k]} | \mu_{t-1}(v,i) - \mu_{t-1}(v,j) | }{2} \right]^2  \right\}.
	\end{align*}
	}
	}
	Output $\hat{\ell}_{\rm wmp}(o) = \argmin_{l \in [k]} |M(o, \bar{t}) - \mu_{\bar{t}}(o, l)|.$}
	\caption{Weighted Message Passing for Multiple Communities}
	\label{alg:wmp-k}
\end{algorithm}
\medskip

\subsection{Vector Evolution \& Concentration}

As in the $k=2$ case,  we establish the recursion formula for the parameter updates. However, unlike the $k=2$ case, for a general initialization $\mu_0$, it is much harder to characterize $\mu_t(u), \sigma_t^2(u)$ analytically, and thus relate the misclassification error to the minimum energy. We will show that this goal can be achieved by a judicious choice of $\mu_0$. We will start with the following Lemma that describes the vector evolution and concentration-of-measure. 

\begin{lemma}[Concentration, general $k$]
	\label{lem:concentration.k}
	Recall the label broadcasting process with Markov transition kernel $K \in \mathbb{R}^{k\times k}$ on tree $T_{\bar{t}}(o)$. Assume the MGF of messages on the leaves $M(u, 0)$ satisfies, for any $\ell \in [k]$
	\begin{align*}
		\mathbb{E}\left[ e^{\lambda M(u,0)} | \ell(u) = l \right] \leq e^{\lambda \mu_0(u, l)} e^{\frac{\lambda^2 \sigma^2_0(u)}{2}}
	\end{align*}
	for any $\lambda$, with parameter
	\begin{align*}
		\mu_0(u) = [\mu_0(u,1),\ldots,\mu_0(u,k)] \in \mathbb{R}^{k}, \quad \sigma^2_0(u) \in \mathbb{R}.
	\end{align*}
	Define the following updating rules for a node $v$
	\begin{align*}
		\mu_t(v) &= \sum_{u \in \C(v)} \theta K \mu_{t-1}(u) \\
		\sigma^2_t(v) &= \sum_{u \in \C(v)} \theta^2 \left\{ \sigma^2_{t-1}(u) + \left[ \frac{ \max\limits_{i,j \in [k]}| \mu_{t-1}(u,i) - \mu_{t-1}(u,j) | }{2} \right]^2  \right\}.
	\end{align*}
	The following concentration-of-measure holds for the root message $M(o, \bar{t})$:
	\begin{align*}
		M(o, \bar{t})_{| \ell(o)=l} \in \mu_{\bar{t}}(o, l) \pm x \cdot \sigma_{\bar{t}}(o)
	\end{align*}
	with probability $1 - 2 \exp(-\frac{x^2}{2})$. In addition, if we we classify the root's label as $$
	\hat{\ell}_{\rm wmp}(o) = \argmin_{l \in [k]} |M(o, \bar{t}) - \mu_{\bar{t}}(o, l)|,
	$$ then the worst-case misclassification error is upper bounded by
	\begin{align}
		\exp (- \frac{\min_{i,j \in [k]}|\mu_{\bar{t}}(o, i) - \mu_{\bar{t}}(o, j)|^2}{8\sigma^2_{\bar{t}}(o)}).
	\end{align}
\end{lemma}

{\begin{remark}
	Unlike the $k=2$ case, in general it is hard to quantitatively analyze this evolution system for $\mu_t(u), \sigma_t^2(u)$. The main difficulty stems from the fact that the coordinates that attain the maximum of
	$\max_{i,j\in [k]} | \mu_{t-1}(u,i) - \mu_{t-1}(u,j) |$ vary with $u, t$. Hence, it is challenging to provide sharp bounds on $\sigma_t^2(u)$. In some sense, the difficulty is introduced by the instability of the relative ordering of the coordinates of the vector $\mu_t(u)$ for an arbitrary initialization.


	As will be shown in the next section, one can resolve this problem by initializing $\mu_0(u, l), l\in [k]$ in a ``most informative'' way. This initialization represents the additional weights on label's types beyond the weights given by the minimum energy flow.
\end{remark}

\subsection{Additional Weighting via Eigenvectors}

We show in this section that the vector evolution system with noisy initialization is indeed tractable if we weigh the label's type according to the second right eigenvector of $K \in \mathbb{R}^{k \times k}$.

\begin{theorem}[Weighting by Eigenvector]
	\label{thm:weight-eigen}
	Assume that the second eigenvalue $\theta = \lambda_2(K)$ of the Markov transition kernel $K$ is a real, and denote the associated second right eigenvector by $w \in \mathbb{R}^k, \| w \| = 1, w^T \mathbf{1} = 0$. Denote the minimum energy flow on tree $T(o)$ with conductance level $\theta^2$ by $\mathbf{i}^*$. In the case of noisy label information with parameter $\delta$, if we initialize
	$$\mu_0(u, l) = \delta \cdot \theta^{-2|u|} \mathbf{i}^*(\rightsquigarrow u) \cdot w_{l}, ~\text{for}~ l \in [k],$$
	and $\sigma_0^2(u) = \left( \theta^{-2|u|} \mathbf{i}^*(\rightsquigarrow u) \right)^2 \cdot \max_{i,j \in [k]} |w_i - w_j|^2$, then the worst case misclassification error is upper bounded by
	\begin{align*}
		\limsup_{t\rightarrow \infty} ~ \max_{i,j \in [k], i\neq j} \mathbb{P}( \hat{\ell}_{\rm wmp}(o) = i | \ell(o) = j ) \leq \exp( - \frac{R^2}{2 \mathbf{E}^*(\theta^{-2})})
	\end{align*}
	with $R = \frac{\min_{i,j} |w_i - w_j|}{\max_{i,j} |w_i - w_j|}$.
\end{theorem}
\begin{remark}
	Observe that the upper bound becomes trivial when $\min_{i,j} |w_i - w_j| = 0$. In this case, one can easily modify in the proof of Theorem~\ref{thm:weight-eigen} so that the following non-trivial guarantee for set misclassification error holds. Assume $w$ has $m$ distinct values, and denote the set $S_i, 1\leq i \leq m$ to be the distinct value sets associated with $w$. Then one has the following upper bound on the set misclassification error
	\begin{align}
		\limsup_{t\rightarrow \infty} ~\max_{i,j \in [m], i\neq j} \mathbb{P}(\hat{\ell}_{\rm wmp}(o) \in S_i | \ell(o) \in S_j ) \leq \exp( - \frac{R_{\rm S}^2}{2 \mathbf{E}^*(\theta^{-2})})
	\end{align}
	with $R_{\rm S} = \frac{\min_{i,j} |w_{S_i} - w_{S_j}|}{\max_{i,j} |w_{S_i} - w_{S_j}|}$.
\end{remark}

\subsection{Lower Bound: Sharp Threshold}

In this section we provide a new lower bound analysis through bounding the $\chi^2$ distance to the ``average measure''. The lower bound shows that the transition boundary $\lambda \theta^2 = 1$ achieved by WMP is sharp for any $k$. To the best of our knowledge, the first lower bound for general $k$ case is achieved in \cite{janson2004robust} using a notion of weighted $\chi^2$ distance. For completeness of the presentation, we provide here a different proof using the usual $\chi^2$ distance. In addition, our approach admits a clear connection to the upper bound analysis through matrix power iterations.

\begin{theorem}[Limitation for Local Algorithms, $k$-communities]
	\label{thm:KL-lower.bound}
	Recall the label broadcasting process with Markov transition kernel $K$ on tree $T_t(o)$. Assume $K \in \mathbb{R}^{k \times k}$ is symmetric. Consider the case when noisy label information (with parameter $\delta$) is known on the depth-$t$ layer leaf nodes.
	Under the condition $$
	{\rm br}[T(o)] \theta^2<1$$ and $k \delta^2(\frac{1}{\delta + \frac{1-\delta}{k}} + \frac{1}{\frac{1-\delta}{k}} ) < 1 - {\rm br}[T(o)] \theta^2$, we have
	$$
	\liminf_{t \rightarrow \infty} ~\inf_{\sigma_t} \max_{l \in [k]} ~\mathbb{P}( \sigma_t(o) \neq \ell(o) | \ell(o) = l ) \geq \frac{1}{2} (1 - \frac{1}{k}).
	$$
	where~ $\sigma_t(o): \tilde{\ell}_{\rm prior}(\C^t(o)) \rightarrow [k]$ is any estimator mapping the prior labels on leaves in the local tree to a decision. The above inequality also implies
	$$
	\liminf_{t \rightarrow \infty} ~\inf_{\sigma_t} \max_{i,j \in [k], i\neq j} ~\mathbb{P}( \sigma_t(o) = i | \ell(o) = j ) \geq \frac{1}{2k}.
	$$
\end{theorem}
The above result shows that even belief propagation suffers the error at least $\frac{1}{2k}$ in distinguishing $i, j$, which is within a factor of $2$ from random guess. We remark in addition that the condition $k \delta^2(\frac{1}{\delta + \frac{1-\delta}{k}} + \frac{1}{\frac{1-\delta}{k}} ) < 1 - {\rm br}[T(o)] \theta^2$ can be satisfied when $\delta$ is small.

\subsection{Further Discussion}

\paragraph{Local versus Global Algorithms}

In the balanced case with $k$ equal size communities, and $p,q$ denoting the within- and between-community connection probabilities, the Kesten-Stigum threshold for local algorithm class takes the following expression
$$
{\sf SNR} := \frac{n(p-q)^2}{k^2(q + \frac{p-q}{k})} = 1.
$$
However, it is known that the limitation for global algorithm class for growing number of communities is ${\sf SNR} \asymp \mathcal{O}(\frac{\log k}{k})$ (\cite{abbe2015detection}, weak consistency) and ${\sf SNR} \asymp \mathcal{O}(\frac{\log n}{k})$ (\cite{chen2014statistical}, strong consistency). Therefore, as $k$ grows, there is an interesting gap between local and global algorithms in terms of ${\sf SNR}$. An interesting direction is to determine whether one can solve the problem down to the information-theoretic threshold $\mathcal{O}^*(\frac{1}{k})$ with computationally efficient algorithms.


\section{Numerical Studies}
\label{sec:num.study}
We apply the message passing Algorithm~\ref{alg:wmp} to the political blog dataset \cite{adamic2005political} (with a total of 1222 nodes) in the partial label information setting with $\delta$ portion randomly revealed labels. In the literature, the state-of-the-art result for a global algorithm appears in \cite{jin2015fast}, where the misclassification rate is $58/1222 = 4.75\%$. Here we run a weaker version of our WMP algorithm as it is much easier to implement and does not require parameter tuning. Specifically, we initialize the message with a uniform flow on leaves (minimum energy flow that corresponds to a regular tree). We will call this algorithm approximate message passing (AMP) within this section.  

We run AMP with three different settings $\delta = 0.1, 0.05, 0.025$, repeating each experiment $50$ times. As a benchmark, we compare the results to the spectral algorithm on the $(1-\delta)n$ sub-network. We focus on the local tree with depth 1 to 5, and output the error for message passing with each depth.  The results are summarized as box-plots in Figure~\ref{fig:polblg}. The left figure illustrates the comparison of AMP with depth 1 to 5 and the spectral algorithm, with red, green, blue  boxes corresponding to $\delta = 0.025, 0.05, 0.1$, respectively. The right figure zooms in on the left plot with only AMP depth 2 to 4 and spectral, to better emphasize the difference.
Remark that if we only look at depth 1, some of the nodes may have no revealed neighbors. In this setting, we classify this node as wrong (this explains why depth-1 error can be larger than 1/2).
\begin{figure}[pht]
	\begin{subfigure}
		{0.49\linewidth} \centering
		\includegraphics[width=\textwidth]{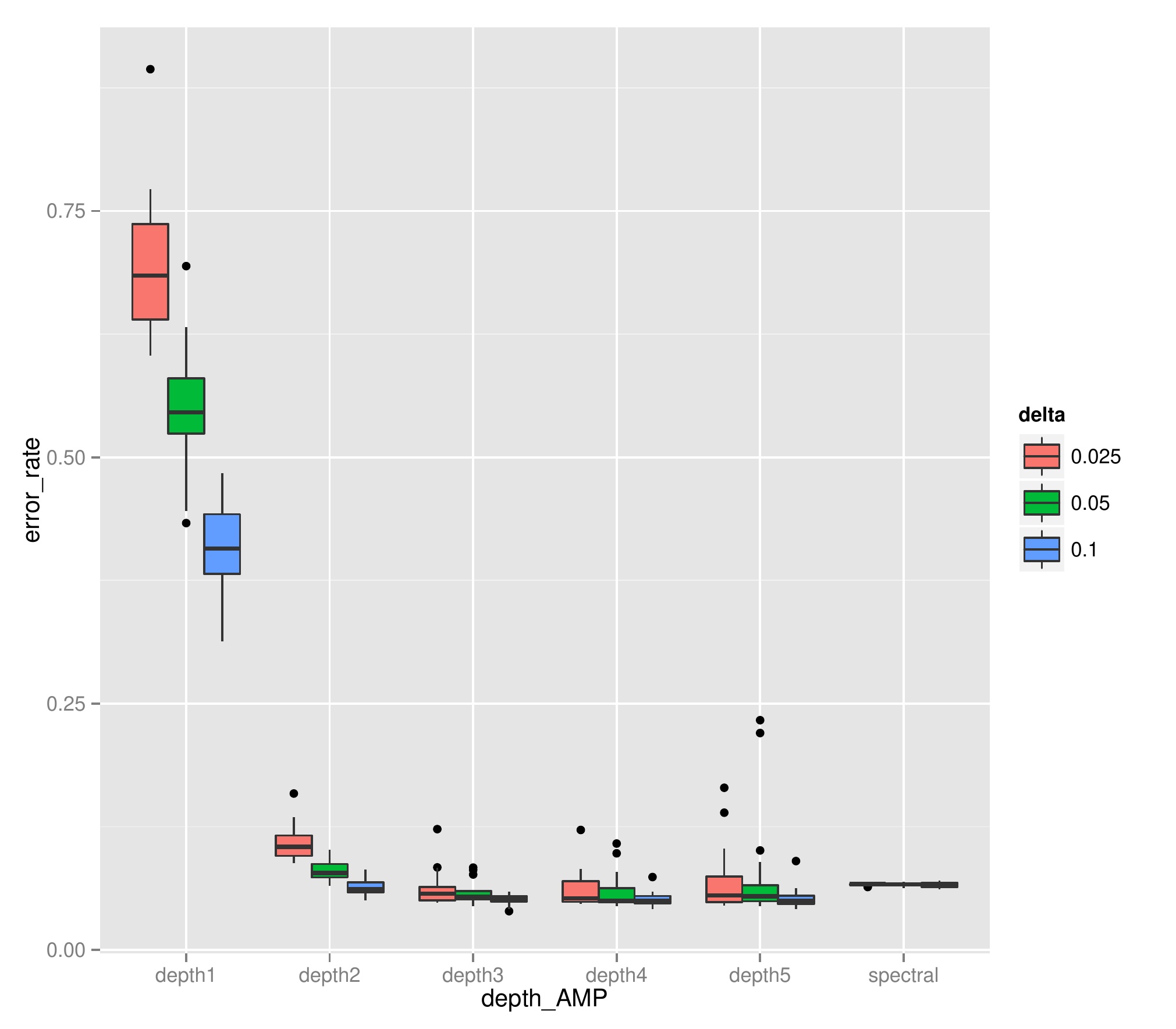}
	\end{subfigure}
	\begin{subfigure}
		{0.49\linewidth} \centering
		\includegraphics[width=\textwidth]{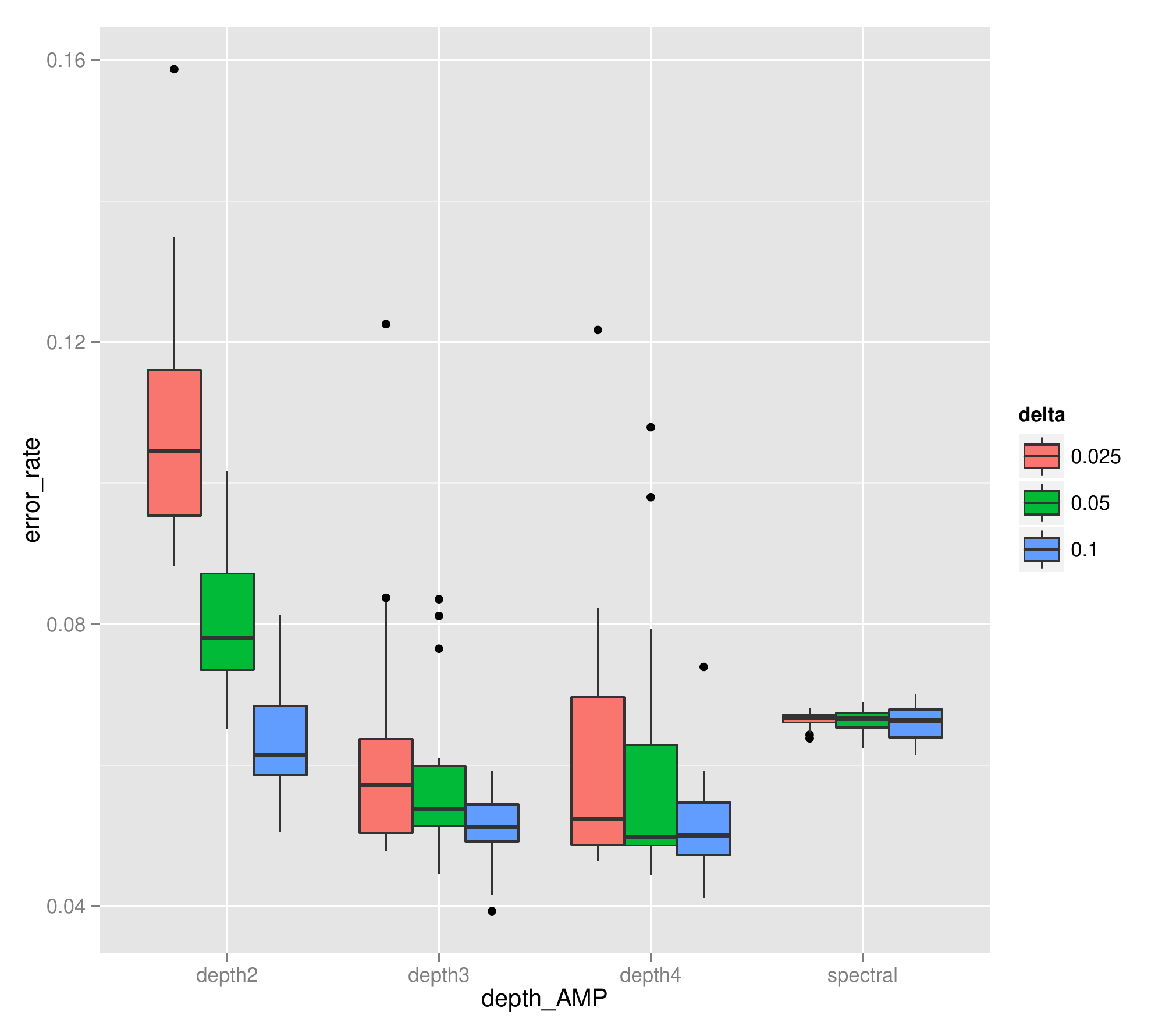}
	\end{subfigure}
	\caption{AMP algorithm on Political Blog Dataset.}
	\label{fig:polblg}
\end{figure}

We present in this paragraph some of the statistics of the experiments, extracted from the above Figure~\ref{fig:polblg}. In the case $\delta = 0.1$,  from depth 2-4, the AMP algorithm produces the mis-classification error rate (we took the median over the experiments for robustness) of $6.31\%, 5.22\%, 5.01\%$, while the spectral algorithm produces the error rate $6.68\%$. When $\delta = 0.05$, i.e. about 60 node labels revealed, the error rates are $7.71\%, 5.44\%, 5.08\%$ with depth 2 to 4, contrasted to the spectral algorithm error $6.66\%$. In a more extreme case $\delta = 0.025$ when there are only $\sim30$ node labels revealed, AMP depth 2-4 has error $10.20\%, 5.71\%, 5.66\%$, while spectral is $6.63\%$. In general, the AMP algorithm with depth 3-4 uniformly beats the vanilla spectral algorithm. Note that our AMP algorithm is a distributed decentralized algorithm that can be run in parallel. We acknowledge that the error $\sim 5\%$ (when $\delta$ is very small) is still slightly worse than the state-of-the-art degree-corrected SCORE algorithm in \cite{jin2015fast}, which is $4.75\%$.

\section{Technical Proofs}
\label{sec:proof}

We will start with two useful results. The first one is a coupling proposition.
The proof follows exactly the same idea as in Proposition 4.2 in \cite{mossel2012stochastic}. The intuition is that
when the depth of the tree is shallow, the SBM in the sparse regime can be coupled to a Galton-Watson tree with Poisson branching (as there are
many nodes outside the radius $R$ for the Poisson-Multinomial coupling, when $R$ small). We want to prove a more general version
for SBM with unequal size communities. The proof is delayed to Appendix~\ref{sec:prop.lma}.

\begin{proposition}
	\label{lma:coupling.tree}
  Let $R = R(n) = \lfloor \frac{1}{4 \log [2 np_0 + 2 \log n]} \log n \rfloor$, where $p_0 = \max_{i,j} Q_{ij}$. 
Denote $(T, \sigma_{T})$ to be the multi-type Galton-Watson tree (with Poisson branching) with mean matrix $Q {\rm diag}(N)$ and label transition kernel $K = \left[ {\rm diag}(QN) \right]^{-1}  Q {\rm diag}(N)$. Denote $G_R$ as the neighborhood of depth up to $R$ induced by the graph $G$, for a particular node. 
There exists a coupling between
  $(G_R, \ell_{G_R})$ and $(T, \sigma_{T})$ such that $(G_R, \ell_{G_R}) = (T_R, \sigma_{T_R})$
  with high probability as $n \rightarrow \infty$. Here the tree equivalence is up to a label
  preserving homomorphism.
\end{proposition}

\begin{lemma}[Hoeffding's Inequality]
	\label{lma:hoeff}
	Let X be any real-valued random variable with expected value $\mathbb{E}X = 0$ and such that $a \leq X \leq b$ almost surely. Then, for all $\lambda>0$,
	$$
	\mathbb{E} \left[ e^{\lambda X} \right] \leq \exp \left( \frac{\lambda^2 (b - a)^2}{8} \right).
	$$
\end{lemma}

\begin{proof}[Proof of Lemma~\ref{lem:k=2.concentration}]
	Recall the linearized message passing rule that ``approximates'' the Bayes optimal algorithm:
	$$M(u,t) =  \sum_{v \in \C(u)} \bar{\theta} \cdot M(v,t-1), ~~\text{where}~ \bar{\theta} = \frac{\theta_1+\theta_2}{2}.$$

	Let us analyze the behavior of the linearized messages $M(u, t)$ for a particular node $u$. The proof follows by induction on $t$. The case $t=0$ follows from the assumption about $\mu_0(u),\sigma^2_0(u)$ and Chernoff bound. Now, assume that the induction premise is true for $t-1$. Note that
	\begin{align*}
	  & \mbb{E}\left[ e^{\lambda M(u, t)} | \ell(u)=+ \right] \\
		&=  \prod_{v\in \C(u)} \mbb{E}\left[ e^{\lambda \bar{\theta} M(v, t-1)} | \ell(u)=+ \right] \\
	  &= \prod_{v\in \C(u)} \left\{ \mbb{E}\left[ e^{\lambda \bar{\theta} M(v, t-1)} | \ell(v)=+ \right] \frac{1+\theta_1}{2}
	  + \mbb{E}\left[ e^{\lambda \bar{\theta} M(v, t-1)} | \ell(v)=- \right] \frac{1-\theta_1}{2} \right\} \\
	  &\leq \prod_{v\in \C(u)} e^{ (\lambda \bar{\theta})^2 \frac{\sigma_{t-1}^2(v)}{2}} \left\{ e^{\lambda \bar{\theta} \mu_{t-1}(v,+)} \frac{1+\theta_1}{2} +
	                                       e^{\lambda \bar{\theta} \mu_{t-1}(v,-)} \frac{1-\theta_1}{2} \right\}\\
	  &\leq \prod_{v\in \C(u)} e^{ (\lambda \bar{\theta})^2 \frac{\sigma_{t-1}^2(v)}{2}}
	          e^{\lambda \bar{\theta} [\mu_{t-1}(v,+)\frac{1+\theta_1}{2} + \mu_{t-1}(v,-)\frac{1-\theta_1}{2} ]}
	          e^{(\lambda \bar{\theta})^2 \frac{ [\mu_{t-1}(v,+) - \mu_{t-1}(v,-)]^2}{8}},
	\end{align*}
	where the last step uses the Hoeffding's Lemma. Rearranging the terms, 
	\begin{align*}
	\mbb{E}\left[ e^{\lambda M(u, t)} | \ell(u)=+ \right] &\leq e^{\lambda \sum_{v \in \C(u)} \bar{\theta} \langle K_{1\cdot}, \mu_{t-1}(v) \rangle} e^{\frac{\lambda^2 \bar{\theta}^2 \sum_{v \in \C(u)} \left\{ \sigma^2_{t-1}(v) + \left[ \frac{\mu_{t-1}(v, +) - \mu_{t-1} (v, -)}{2} \right]^2  \right\}}{2} } \\
	&= e^{\lambda \mu_{t}(u, +)} e^{\frac{\lambda^2 \sigma^2_{t}(u)}{2}},
	\end{align*}
	where $K_{1\cdot}$ denotes the first row of transition matrix $K$.
	Clearly, same derivation holds with $\ell(u) = -$. Applying the Chernoff bound and optimizing over $\lambda$, one arrives at the exponential concentration bound. Induction completes.

	To upper bound the misclassification error, simply plug in the standardized absolute values of the difference, namely $x = \left| \frac{\mu_{\bar{t}}(o,+) - \mu_{\bar{t}}(o,-)}{2 \sigma_{\bar{t}}(o)} \right|$. 
\end{proof}
\begin{remark}
	Now let us propose the choice of $\mu_0(u)$ and $\sigma^2_0(u)$ for the case of noisy label information with parameter $\delta$. In WMP algorithm, choose $ M(u, 0) = c(u) \sgn(\tilde{\ell}_{\rm prior})$ with factor $c(u)$ that depends on the node $u$. Using simple Hoeffding's concentration for Bernoulli r.v., one has
	\begin{align*}
		&\mu_0(u, +) = c(u) \delta,~ \mu_0(u, -) = - c(u) \delta, \\
		&\text{and}~ \sigma^2_{0}(u) = c(u)^2.
	\end{align*}
\end{remark}

\begin{proof}[Proof of Theorem~\ref{thm:energy-error}]
	Using the result of Lemma~\ref{lem:k=2.concentration}, the proof analyzes evolution of 
	$$\frac{\sigma^2_{t}(o)}{\left[ \frac{[\mu_{t}(o, +) - \mu_{t}(o, -)]}{2} \right]^2}.$$
	 First, let us derive the expression for $\mu_{t}(o, +) - \mu_{t}(o, -)$. Denoting $w = [1, -1]^T$, it is easy to verify that $w^T K = \bar{\theta} w^T$. We have,
	\begin{align*}
		\mu_{t}(o, +) - \mu_{t}(o, -) &= \sum_{v \in \C(o)} \bar{\theta} w^T K \mu_{t-1}(v) =\sum_{v \in \C(o)} \bar{\theta}^2 w^T \mu_{t-1}(v) \\
		& = \bar{\theta}^2 \sum_{v \in \C(o)} [\mu_{t-1}(v, +) - \mu_{t-1}(v, -)].
	\end{align*}
	Using the above equation recursively, one can easily see that for any $d, 1\leq d \leq t$,
	\begin{align}
		\label{eq:mean-formula}
		\mu_{t}(o, +) - \mu_{t}(o, -) &= \bar{\theta}^{2d} \sum_{v \in \C^d(o)} [\mu_{t-d}(v, +) - \mu_{t-d}(v, -)].
	\end{align}
	Now for $\sigma^2_{t}(o)$ for $\sigma^2_t(\rho)$, one has
  \begin{align}
    \sigma_t^2(o) &= \bar{\theta}^2 \sum_{v \in \C(o)} \left\{ \sigma_{t-1}^2(v) + \left[ \frac{\mu_{t-1}(v, +) - \mu_{t-1}(v, -)}{2} \right]^2 \right\} \nonumber
	\end{align}
	which can be written, in turn, as
	\begin{align}
    &\bar{\theta}^2 \sum_{v \in \C(o)}  \bar{\theta}^2 \sum_{u \in \C(v)} \left\{ \sigma_{t-2}^2(u) + \left[ \frac{\mu_{t-2}(u, +) - \mu_{t-2}(u, -)}{2} \right]^2 \right\} \nonumber
    \\
    & \quad \quad + \bar{\theta}^2 \sum_{v \in \C(\rho)} \left[ \frac{\mu_{t-1}(v, +) - \mu_{t-1}(v, -)}{2} \right]^2 \nonumber\\
    & = \ldots \ldots +\bar{\theta}^4 \sum_{v \in \C(\rho)} \sum_{u \in \C(v)} \left[ \frac{\mu_{t-2}(u, +) - \mu_{t-2}(u, -)}{2} \right]^2 + \bar{\theta}^2 \sum_{v \in \C(o)} \left[ \frac{\mu_{t-1}(v, +) - \mu_{t-1}(v, -)}{2} \right]^2 \nonumber \\
    & = \sum_{v \in T_t(o)} \bar{\theta}^{2|v|} \left[ \frac{\mu_{t-|v|}(v, +) - \mu_{t-|v|}(v, -) }{2} \right]^2 + \sum_{u \in \C^{t}(o)}  \bar{\theta}^{2t} \sigma^2_{0}(u) \nonumber .
  \end{align}
	Using the above equation one can bound
	\begin{align}
		\frac{\sigma^2_{t}(o)}{\left[ \frac{[\mu_{t}(o, +) - \mu_{t}(o, -)]}{2} \right]^2} &= \frac{\sum_{v \in T_t(o)} \bar{\theta}^{2|v|} \left[ \frac{\mu_{t-|v|}(v, +) - \mu_{t-|v|}(v, -) }{2} \right]^2}{\left[ \frac{[\mu_{t}(o, +) - \mu_{t}(o, -)]}{2} \right]^2} + \frac{\sum_{u \in \C^{t}(o)}  \bar{\theta}^{2t} \sigma^2_{0}(u)}{\left[ \frac{[\mu_{t}(o, +) - \mu_{t}(o, -)]}{2} \right]^2} \nonumber \\
		&= \sum_{v \in T_t(o)} \frac{\bar{\theta}^{2|v|} \left[ \mu_{t-|v|}(v, +) - \mu_{t-|v|}(v, -)  \right]^2}{\left[ [\mu_{t}(o, +) - \mu_{t}(o, -)] \right]^2} + R \nonumber \\
		& = \sum_{v \in T_t(o)} \frac{ \left( \bar{\theta}^{2|v|} [\mu_{t-|v|}(v, +) - \mu_{t-|v|}(v, -)]  \right)^2}{\left( [\mu_{t}(o, +) - \mu_{t}(o, -)] \right)^2} \bar{\theta}^{-2|v|} + R \label{eq:error.energy}
	\end{align}
	where the remainder $$R = \frac{\sum_{u \in \C^{t}(o)}  \bar{\theta}^{2t} \sigma^2_{0}(u)}{\left[ \frac{[\mu_{t}(o, +) - \mu_{t}(o, -)]}{2} \right]^2}.$$
	Recall the definition of $$\mathbf{i}(\rightsquigarrow v) = \frac{  \bar{\theta}^{2|v|} [\mu_{t-|v|}(v, +) - \mu_{t-|v|}(v, -)]  }{ [\mu_{t}(o, +) - \mu_{t}(o, -)] }.$$ It is clear from Eq.\eqref{eq:mean-formula} that $\mathbf{i}$ is a valid unit flow, in the sense of Definition~\ref{def:unit_flow}.
	Continuing with Eq.~\eqref{eq:error.energy}, one has
	\begin{align}
		\inf_{\mathbf{i}}~ \frac{\sigma^2_{t}(o)}{\left[ \frac{[\mu_{t}(o, +) - \mu_{t}(o, -)]}{2} \right]^2} &\leq  \sum_{v \in T_t(o)} \mathbf{i}^*(\rightsquigarrow v)^2 \bar{\theta}^{-2|v|} + R \nonumber \\
		& = \mathbf{E}_{t}(\mathbf{i}^*, \bar{\theta}^{-2}) + R. \label{eq:min-flow-choice}
	\end{align}
	Let us now estimate $R$:
	\begin{align*}
		R &= \frac{\sum_{u \in \C^{t}(o)}  \bar{\theta}^{2t} \sigma^2_{0}(u)}{\left[ \frac{[\mu_{t}(o, +) - \mu_{t}(o, -)]}{2} \right]^2} \\
		&\leq  \sum_{u \in \C^{t}(o)} \mathbf{i}^*(\rightsquigarrow u)^2 \bar{\theta}^{-2t} \cdot  \max_{u \in \C^{t}(o)} \frac{\sigma_0^2(u)}{\left[ \frac{[\mu_{0}(u, +) - \mu_{0}(u,-)]}{2} \right]^2 } \\
		& = \sum_{u \in \C^{t}(o)} \mathbf{i}^*(\rightsquigarrow u)^2 \bar{\theta}^{-2t} \frac{1}{\delta^2}.
	\end{align*}
	The last step is because for noisy label information with parameter $\delta$, 
	$$\frac{\sigma_0^2(u)}{\left[ \frac{[\mu_{0}(u, +) - \mu_{0}(u,-)]}{2} \right]^2 } = \frac{1}{\delta^2}.$$
	In the case when $\lim_{t\rightarrow \infty} \mathbf{E}_{t}(\mathbf{i}^*, \bar{\theta}^{-2}) < \infty$, we know $ \sum_{u \in \C^{t}(o)} \mathbf{i}^*(\rightsquigarrow u)^2 \bar{\theta}^{-2t}  = \mathbf{E}_{t}(\mathbf{i}^*, \bar{\theta}^{-2}) - \mathbf{E}_{t-1}(\mathbf{i}^*, \bar{\theta}^{-2}) \rightarrow 0$. Therefore, $R = \frac{1}{\delta^2} o_t(1).$

	Going back to Eq.~\eqref{eq:min-flow-choice}, to minimize the LHS (ratio between noise and signal), one needs to make sure that $\mathbf{i} = \mathbf{i}^*$, the minimum energy flow. Therefore, the optimal strategy is to initialize $\mu_0(u)$ according to $\mathbf{i}^*(\rightsquigarrow u)$. Thus, if we choose
	$$ \mu_0(u, +) = \delta \bar{\theta}^{-2|u|} \mathbf{i}^*(\rightsquigarrow u), \mu_0(u, -) = \delta \bar{\theta}^{-2|u|} \mathbf{i}^*(\rightsquigarrow u),$$
	we obtain
	\begin{align*}
		\lim_{t \rightarrow \infty} \inf_{\mathbf{i}}~ \frac{\sigma^2_{t}(o)}{\left[ \frac{[\mu_{t}(o, +) - \mu_{t}(o, -)]}{2} \right]^2} = \mathbf{E}^*(\bar{\theta}^{-2}).
	\end{align*}
	From Definition~\ref{def:branch.num},
	\begin{align*}
		\mathbf{E}^*(\bar{\theta}^{-2}) < \infty ~~\text{iff}~~ \bar{\theta}^{-2} < {\rm br}[T(o)].
	\end{align*}
\end{proof}

\begin{proof}[Proof of Theorem~\ref{thm:weight-eigen}]
	Note that by Perron-Frobenius Theorem, we have $|\theta| = |\lambda_2(K)|<1$. Thanks to the choice of $w$,
	$$\mathbb{E}[M_0(u)| \ell(u) = l] = \delta \theta^{-2|u|} \mathbf{i}^*(\rightsquigarrow u) w_l + \frac{1-\delta}{k} \theta^{-2|u|} \mathbf{i}^*(\rightsquigarrow u) w^T \mathbf{1}  = \delta \theta^{-2|u|} \mathbf{i}^*(\rightsquigarrow u) w_l.$$

	Let us first derive the formula for $\mu_t(o) \in \mathbb{R}^k$ under the chosen initialization $\mu_0(u)$. We claim that
	$$ \mu_{t - |v|}(v) = \delta \cdot \theta^{-2|v|} \mathbf{i}^*(\rightsquigarrow v) \cdot w.$$
	Proof is via induction. The base case $|u| = t$ is exactly the choice of the initialization. Let us assume for $|u| > |v|$ the claim is true, and prove for $v$:
	\begin{align*}
		\mu_{t-|v|}(v) &= \sum_{u \in \C(v)} \theta K \mu_{t-1}(u) \\
		&= \sum_{u \in \C(v)} \theta Kw  \cdot \delta \theta^{-2|v|-2} \mathbf{i}^*(\rightsquigarrow u) \\
		& = \sum_{u \in \C(v)} \theta^2 w \cdot \delta \theta^{-2|v|-2} \mathbf{i}^*(\rightsquigarrow v) =  \delta \cdot \theta^{-2|v|} \mathbf{i}^*(\rightsquigarrow v) \cdot w,
	\end{align*}
	completing the induction.

	Now let us bound $\sigma^2_t(o)$. Observe that in our derived formula for $\mu_{t-|v|}(v)$, all the coordinates are proportional to $w$. In other words, $\mu_{t-|v|}(v)$ stays in the direction of $w$ for all $v$. This greatly simplifies the expression for $\sigma^2_t(o)$. We have
	\begin{align*}
		\sigma^2_t(o) &= \sum_{v \in T_t(o)} \theta^{2|v|} \left[ \frac{\max_{i,j \in [k]} |\mu_{t-|v|}(v, i) - \mu_{t-|v|}(v, j)| }{2} \right]^2 + \sum_{u \in \C^{t}(o)}  \theta^{2t} \sigma^2_{0}(u) \\
		&= \delta^2 \left[ \frac{\max_{i,j \in [k]} |w(i) - w(j)| }{2} \right]^2 \sum_{v \in T_t(o)} \mathbf{i}^*(\rightsquigarrow v)^2 \theta^{-2|v|} \\
		& \quad \quad +  \left[ \frac{\max_{i,j \in [k]} |w(i) - w(j)| }{2} \right]^2 \sum_{v \in \C^t(o)} \mathbf{i}^*(\rightsquigarrow v)^2 \theta^{-2|v|}.
	\end{align*}
	Plugging in the definition $R = \frac{\min_{i,j} |w_i - w_j|}{\max_{i,j} |w_i - w_j|}$, under the condition $${\rm br}[T(o)] \theta^2 > 1,$$ we have $\mathbf{E}(\mathbf{i}^*, \theta^{-2}) < \infty$, and
	\begin{align*}
		\frac{\sigma^2_{\bar{t}}(o)}{ \left[ \frac{\min_{i,j \in [k]}|\mu_{\bar{t}}(o, i) - \mu_{\bar{t}}(o, j)|}{2} \right]^2 } & = \frac{1}{R^2} \mathbf{E}(\mathbf{i}^*, \theta^{-2}) + \frac{1}{\delta^2 R^2} o_t(1).
	\end{align*}

\end{proof}

\begin{proof}[Proof of Theorem~\ref{thm:KL-lower.bound}]
	Recall that $\pi( \ell_{\partial T_t(o) \cap T_{t-|u|}(u)} | \ell(u) = i)$ denotes the probability measure on the leaf labels on depth $t$, given $\ell(u) = i$. For a node $u$, when there is no confusion, we abbreviate the measure $\pi( \ell_{\partial T_t(o) \cap T_{t-|u|}(u)}  | \ell(u) = i)$ as $\pi_u(i)$. According to Perron-Frobenius Theorem, there is a unique left eigenvector for $K$ with eigenvalue $1$, denote this by $w \in \mathbb{R}^k$. Under the assumption $K$ being symmetric, we know that $w = \frac{1}{k} \mathbf{1}$.
	 Denote $\bar{\pi}_u =  \sum_{j=1}^k w(j) \pi_u(j)$.

	Let us bound the $d_{\chi^2}\left( \pi_u(i) || \bar{\pi}_u \right)$ by deriving a recursive bound:
	\begin{align*}
		\log \left[ 1 + d_{\chi^2}\left( \pi_u(i) || \bar{\pi}_u \right) \right] &= \sum_{v \in \C(u)} \log \left[ 1 + d_{\chi^2}\left( \sum_{l=1}^k K_{il} \pi_v(l) || \sum_{j=1}^k \sum_{l=1}^k w(j) K_{jl} \pi_v(l) \right) \right] \\
		&= \sum_{v \in \C(u)} \log \left[ 1 + d_{\chi^2}\left( \sum_{l=1}^k K_{il} \pi_v(l) || \bar{\pi}_v \right) \right] 
	\end{align*}
	since $w^T K = w^T$. By definition, the above expression is
	\begin{align*}
		&\sum_{v \in \C(u)} \sum_{v \in \C(u)} \log \left[ 1 + \int \frac{\left[ \sum_{l=1}^k K_{il} \pi_v(l) - \bar{\pi}_v \right]^2}{\bar{\pi}_v} \right] \\
		& = \sum_{v \in \C(u)} \sum_{v \in \C(u)} \log \left[ 1 + \int \frac{\left[ \sum_{l=1}^k K_{il} (\pi_v(l) - \bar{\pi}_v) \right]^2}{\bar{\pi}_v} \right] \\
		& \leq  \sum_{v \in \C(u)}  \int \frac{\left[ \sum_{l=1}^k K_{il} (\pi_v(l) - \bar{\pi}_v) \right]^2}{\bar{\pi}_v}.
	\end{align*}
	Now we know that
	\begin{align*}
	\sum_{i=1}^k \log \left[ 1 + d_{\chi^2}\left( \pi_u(i) || \bar{\pi}_u \right) \right] &\leq \sum_{v \in \C(u)}  \int \sum_{i=1}^k \frac{\left[ \sum_{l=1}^k K_{il} (\pi_v(l) - \bar{\pi}_v) \right]^2}{\bar{\pi}_v}.
	\end{align*}
	Recall the following fact that for any $z_1, z_2, \ldots z_k \geq 0$,
	$$\log(1 + \sum_{i=1}^k z_i) \leq \sum_{i=1}^k \log(1 +  z_i).$$ Using this fact the lower bound the LHS, we reach
	\begin{align*}
	\log \left[ 1 + \sum_{i=1}^k  d_{\chi^2}\left( \pi_u(i) || \bar{\pi}_u \right) \right] &\leq \sum_{v \in \C(u)}  \int \sum_{i=1}^k \frac{\left[ \sum_{l=1}^k K_{il} (\pi_v(l) - \bar{\pi}_v) \right]^2}{\bar{\pi}_v}\\
	 & \leq \sum_{v \in \C(u)}  \int \frac{\| K (\pi_v(\cdot) - \bar{\pi}_v \mathbf{1}) \|^2}{\bar{\pi}_v} \\
	 & \leq  \theta^2 \sum_{v \in \C(u)}  \int \frac{\|(\pi_v(\cdot) - \bar{\pi}_v \mathbf{1}) \|^2}{\bar{\pi}_v}  \\
	 & = \theta^2 \sum_{v \in \C(u)} \sum_{i=1}^k d_{\chi^2}\left( \pi_v(i) || \bar{\pi}_v \right)
	\end{align*}
	where the last two lines use the fact that $\pi_v(\cdot) - \bar{\pi}_v \mathbf{1} \perp \mathbf{1}$, therefore $\| K (\pi_v(\cdot) - \bar{\pi}_v \mathbf{1}) \|^2 \leq \theta^2 \| \pi_v(\cdot) - \bar{\pi}_v \mathbf{1} \|^2$.
	
	We will need the the following Lemma that describes the branching number through the cutset.
	\begin{lemma}[\cite{pemantle1999robust}, Lemma 3.3]
		\label{lem:cutset}
		Assume $ {\rm br}[T] < \lambda$. Then for all $\epsilon>0$, there exists a cutset $C$ such that
		\begin{align}
		\label{eq:cutset.all}
		\sum_{x \in C} \left( \frac{1}{\lambda} \right)^{|x|} \leq \epsilon
		\end{align}
		and for all $v$ such that $|v| \leq \max_{x \in C} |x|$,
		\begin{align}
		\label{eq:cutset.partial}
		\sum_{x \in C \cap T(v)} \left( \frac{1}{\lambda} \right)^{|x|-|v|} \leq 1.
		\end{align}
		Here the notation $|v|$ denotes the depth of $v$.
	\end{lemma}
	Let us use the cutset argument to prove $ \sum_{i=1}^k  d_{\chi^2}\left( \pi_u(i) || \bar{\pi}_u \right)  \rightarrow 0$ when $|u| \rightarrow \infty$.
	Fix any $\lambda$ such that $\theta^{-2} > \lambda >{\rm br}[T(o)]$.
	For any $\epsilon$ small, the above Lemma claims the existence of cutset $C_\epsilon$ such that Eq.~\eqref{eq:cutset.all} and \eqref{eq:cutset.partial} hold.
	Let us prove through induction on $\max_{x \in C_\epsilon} |x| - |v|$ that for any $v$ such that $|v| \leq \max_{x \in C_\epsilon} |x|$, we have
	\begin{align}
	\label{eq:induction}
	\sum_{i=1}^k d_{\chi^2}\left( \pi_v(i) || \bar{\pi}_v \right) \leq \eta \sum_{x \in C_{\epsilon} \cap T(v)} \left( \frac{1}{\lambda} \right)^{|x|-|v|} \leq \eta.
	\end{align}
	with the choice $\eta = k \delta^2(\frac{1}{\delta + \frac{1-\delta}{k}} + \frac{1}{\frac{1-\delta}{k}} )$.
	First for the base case, the claim is true because of the choice of $\eta$.

	Preceding with the induction, assume for $v$ such that $\max_{x \in C_\epsilon} |x| - |v| = t-1$ equation~\eqref{eq:induction} is satisfied, and let us prove for $v: \max_{x \in C_\epsilon} |x| - |u| = t$. We recall the linearized recursion
	\begin{align*}
		\log \left[ 1 + \sum_{i=1}^k  d_{\chi^2}\left( \pi_u(i) || \bar{\pi}_u \right) \right] & \leq \theta^2 \sum_{v \in \C(u)} \sum_{i=1}^k d_{\chi^2}\left( \pi_v(i) || \bar{\pi}_v \right) \\
		& \leq \theta^2 \sum_{v \in \C(u)}  \eta \sum_{x \in C_{\epsilon} \cap T(v)} \left( \frac{1}{\lambda} \right)^{|x|-|v|} \\
		& =   \theta^2 \lambda \cdot \eta  \sum_{x \in C_{\epsilon} \cap T(u)} \left( \frac{1}{\lambda} \right)^{|x|-|u|} 
\end{align*}
 Using the assumption $\theta^2 \lambda < \frac{1}{1+\eta}$, the above can be upper bounded by
	\begin{align*}
		&\frac{\eta \sum_{x \in C_{\epsilon} \cap T(u)} \left( \frac{1}{\lambda} \right)^{|x|-|u|}}{ 1 + \eta} \quad 
		\leq \frac{\eta \sum_{x \in C_{\epsilon} \cap T(u)} \left( \frac{1}{\lambda} \right)^{|x|-|u|}}{ 1 + \eta \sum_{x \in C_{\epsilon} \cap T(u)} \left( \frac{1}{\lambda} \right)^{|x|-|u|}}
	\end{align*}
	where the last inequality uses the fact that $\sum_{x \in C_{\epsilon} \cap T(u)} \left( \frac{1}{\lambda} \right)^{|x|-|u|} < 1$. Now we know that
	\begin{align*}
		\frac{\sum_{i=1}^k  d_{\chi^2}\left( \pi_u(i) || \bar{\pi}_u \right)}{1 + \sum_{i=1}^k  d_{\chi^2}\left( \pi_u(i) || \bar{\pi}_u \right)} &\leq \log \left[ 1 + \sum_{i=1}^k  d_{\chi^2}\left( \pi_u(i) || \bar{\pi}_u \right) \right] \\
		& \leq \frac{\eta \sum_{x \in C_{\epsilon} \cap T(u)} \left( \frac{1}{\lambda} \right)^{|x|-|u|}}{ 1 + \eta \sum_{x \in C_{\epsilon} \cap T(u)} \left( \frac{1}{\lambda} \right)^{|x|-|u|}}.
	\end{align*}
	By monotonicity of $x/(1+x)$ we have proved the induction claim holds as
	$$
	\sum_{i=1}^k  d_{\chi^2}\left( \pi_u(i) || \bar{\pi}_u \right) \leq \eta \sum_{x \in C_{\epsilon} \cap T(u)} \left( \frac{1}{\lambda} \right)^{|x|-|u|}.
	$$
	Take $\epsilon \rightarrow 0, \lambda \rightarrow {\rm br}[T(o)]$. Define $t_\epsilon: = \min\{ |x|, x\in C_{\epsilon} \}$, it is also easy to see from equation~\eqref{eq:cutset.all} that
	$$
	\left( \frac{1}{\lambda} \right)^{t_\epsilon} \leq \sum_{x \in C_\epsilon} \left( \frac{1}{\lambda} \right)^{|x|} \leq \epsilon \Rightarrow t_{\epsilon} > \frac{\log(1/\epsilon)}{\log \lambda} \rightarrow \infty.
	$$
	Putting things together, under the condition
	\begin{align*}
	 \eta \leq 1 - {\rm br}[T(o)] \theta^2,
	\end{align*}
	we have
	$$\lim_{t \rightarrow \infty} \frac{1}{k} \sum_{i=1}^k  d_{\chi^2}\left( \pi_u(i) || \bar{\pi}_u \right)  \leq \frac{\eta}{k} \cdot \lim_{\epsilon \rightarrow 0} \sum_{x \in C_{\epsilon} \cap T(o)} \left( \frac{1}{\lambda} \right)^{|x|} = 0.$$

	Finally, we invoke the multiple testing argument Theorem 2.6 in \cite{tsybakov2009introduction}).
	\begin{lemma}[\cite{tsybakov2009introduction}, Proposition 2.4, Theorem 2.6]
		\label{lma:multi-test}
		Let $P_0, P_1,\ldots, P_{k}$ be probability measures on $(\mathcal{X}, \mathcal{A})$ satisfying
		$$
		\frac{1}{k}\sum_{i=1}^{k} d_{\chi^2}(P_j, P_0) \leq  k \alpha_*
		$$
		then we have for any selector $\psi:\mathcal{X} \rightarrow [k]$
		$$
		\max_{i \in [k]} P_i(\psi \neq i) \geq \frac{1}{2}(1 - \alpha_* - \frac{1}{k}).
		$$
	\end{lemma}
	Plugging in the result with $P_0 = \bar{\pi}_o$ and $P_i = \pi_o(i)$, we conclude that
	$$
	\liminf_{t \rightarrow \infty} ~\inf_{\sigma} \max_{l \in [k]} ~\mathbb{P}( \sigma(o) \neq \ell(o) | \ell(o) = l ) \geq \frac{1}{2} (1 - \frac{1}{k}).
	$$

\end{proof}

\begin{proof}[Proof of Theorem~\ref{thm:k=2.statement}]
	Given Proposition~\ref{lma:coupling.tree}, Theorem~\ref{thm:energy-error} and Theorem~\ref{thm:k=2.lower.bound}, the proof of Theorem~\ref{thm:k=2.statement} is simple.
	By Proposition~\ref{lma:coupling.tree}, one can couple the local neighborhood of SBM with multi-type Galton Watson process asymptotically almost surely as $n \rightarrow \infty$, where the label transition matrix is
	$$K :=
		\begin{bmatrix}
		\frac{n_1 Q_{11}}{n_1 Q_{11} + n_2 Q_{12}} & \frac{n_2 Q_{12}}{n_1 Q_{11} + n_2 Q_{12}}\\
		\frac{n_1 Q_{21}}{n_1 Q_{21} + n_2 Q_{22}} & \frac{n_2 Q_{22}}{n_1 Q_{21} + n_2 Q_{22}}\\
		\end{bmatrix}.$$
	For the upper bound, Theorem~\ref{thm:energy-error} shows that the misclassification error is upper bounded by $\exp \left( - \frac{1}{\mathbf{E}^* (\bar{\theta}^{-2})} \right)$ as the depth of the tree goes to infinity. Note if we first send $n\rightarrow \infty$, due to Proposition~\ref{lma:coupling.tree}, the coupling is valid even when $R\rightarrow \infty$ with a slow rate $\log n/\log \log n$. Therefore, the upper bound on misclassification error holds. One can establish the lower bound using the same argument together with Theorem~\ref{thm:k=2.lower.bound}. Finally, for the expression on transition boundary, we know that condition on non-extinction, the branching number for this coupled multi-type Galton Watson tree is $\lambda_1(Q {\rm diag}(N))$ almost surely. Proof is completed.
\end{proof}

\section{Additional Proofs}
\label{sec:prop.lma}
\begin{proof}[Proof of Lemma~\ref{lem:concentration.k}]
The proof logic here is similar to the $k=2$ case.
Again, we analyze the message $M(u, t)$ for a particular node $u$. Use induction on $t$ for the claim
$$
\mathbb{E}\left[ e^{\lambda M(u,t)} | \ell(u) = l \right] \leq e^{\lambda \mu_t(u, l)} e^{\frac{\lambda^2 \sigma^2_t(u)}{2}}.$$

The case for $t=0$ follows from the assumption about $\mu_0(u),\sigma^2_0(u)$ and Chernoff bound.

Assume that the induction is true for $t-1$, and prove the case for $t$. Note that
\begin{align*}
	& \mbb{E}\left[ e^{\lambda M(u, t)} | \ell(u)=l \right] \\
	&=  \prod_{v\in \C(u)} \mbb{E}\left[ e^{\lambda \theta M(v, t-1)} | \ell(u)=l \right] \\
	&= \prod_{v\in \C(u)} \left\{ \sum_{i = 1}^k \mbb{E}\left[ e^{\lambda \theta M(v, t-1)} | \ell(v)=i \right] K_{li} \right\} \\
	&\leq \prod_{v\in \C(u)} e^{ (\lambda \theta)^2 \frac{\sigma_{t-1}^2(v)}{2}} \left\{ \sum_{i = 1}^k e^{\lambda \theta \mu_{t-1}(v, i)} K_{l i}  \right\}\\
	&\leq \prod_{v\in \C(u)} e^{ (\lambda \bar{\theta})^2 \frac{\sigma_{t-1}^2(v)}{2}}
					e^{\lambda \theta [\sum_{i=1}^k \mu_{t-1}(v,i) K_{l i} ]}
					e^{(\lambda \theta)^2 \frac{ \max_{i,j\in [k]}|\mu_{t-1}(v,i) - \mu_{t-1}(v,j)|^2}{8}},
\end{align*}
where the last step uses the Hoeffding's Lemma. Rearrange the terms, one can see that the above equation implies
\begin{align*}
\mbb{E}\left[ e^{\lambda M(u, t)} | \ell(u)=l \right] &\leq e^{\lambda \sum_{v \in \C(u)} \theta \langle K_{l\cdot}, \mu_{t-1}(u) \rangle} e^{\frac{\lambda^2 \theta^2 \sum_{v \in \C(u)} \left\{ \sigma^2_{t-1}(v) + \max_{i,j \in [k]}\left| \frac{\mu_{t-1}(v, +) - \mu_{t-1} (v, -)}{2} \right|^2  \right\}}{2} } \\
&= e^{\lambda \mu_{t}(u, l)} e^{\frac{\lambda^2 \sigma^2_{t}(u)}{2}},
\end{align*}
where $K_{l\cdot}$ denotes the $l-$row of transition matrix $K$.
Apply the Chernoff bound to optimize over $\lambda$, one can arrive the exponential concentration bound. Induction completes.

To upper bound the misclassification error, simply plug in $$|x| =  \frac{\min_{i,j \in [k]} |\mu_{\bar{t}}(o,i) - \mu_{\bar{t}}(o,j)|}{2 \sigma_{\bar{t}}(o)}.$$
\end{proof}

\begin{proof}[Proof of Theorem~\ref{thm:k=2.lower.bound}]
We will gave the proof of Theorem~\ref{thm:k=2.lower.bound} (for the $\delta$ noisy label information case) here.

Define the measure $\pi_{\ell_{T_{t}(o)}}^{+}$ on the revealed labels, for a depth $t$ tree rooted from $o$ with label $\ell(o) = +$ (and similarly define $\pi_{\ell_{T_{t}(o)}}^{-}$). We have the following recursion formula
\begin{align*}
\pi_{\ell_{T_{t}(o)}}^{+} = \prod_{v \in \C(o)} \left[ \frac{1+\theta_1}{2} \pi_{\ell_{T_{t-1}(v)}}^{+} + \frac{1-\theta_1}{2} \pi_{\ell_{T_{t-1}(v)}}^{-}  \right].
\end{align*}
Recall that the $\chi^2$ distance between two absolute continuous measures $\mu(x),\nu(x)$ is
$d_{\chi^2}(\mu,\nu) = \int \frac{\mu^2}{\nu} dx - 1,$
and we have the total variation distance between these two measures is upper bounded by the $\chi^2$ distance
$d_{\rm TV} \left( \mu,\nu \right) \leq \sqrt{d_{\chi^2} \left( \mu,\nu \right) }.$

Let us upper bound the symmetric version of $\chi^2$ distance defined as
$$D_{T_{t}(o)} := \max \left\{ d_{\chi^2} \left( \pi_{\ell_{T_{t}(o)}}^{+},\pi_{\ell_{T_{t}(o)}}^{-} \right) , d_{\chi^2} \left( \pi_{\ell_{T_{t}(o)}}^{-},\pi_{\ell_{T_{t}(o)}}^{+} \right)  \right\}$$ (abbreviate as $D_t(o)$ when there is no confusion), we have the following recursion
\begin{align*}
&\quad \log \left[1+ d_{\chi^2} \left( \pi_{\ell_{T_{t}(o)}}^{+},\pi_{\ell_{T_{t}(o)}}^{-} \right)\right] \\
&= \sum_{v \in \C(o)} \log \left[1+  d_{\chi^2} \left(\frac{1+\theta_1}{2} \pi_{\ell_{T_{t-1}(v)}}^{+} + \frac{1-\theta_1}{2} \pi_{\ell_{T_{t-1}(v)}}^{-}, \frac{1-\theta_2}{2} \pi_{\ell_{T_{t-1}(v)}}^{+} + \frac{1+\theta_2}{2} \pi_{\ell_{T_{t-1}(v)}}^{-} \right) \right]
\end{align*}
\begin{align*}
	& \quad d_{\chi^2} \left(\frac{1+\theta_1}{2} \pi_{\ell_{T_{t-1}(v)}}^{+} + \frac{1-\theta_1}{2} \pi_{\ell_{T_{t-1}(v)}}^{-}, \frac{1-\theta_2}{2} \pi_{\ell_{T_{t-1}(v)}}^{+} + \frac{1+\theta_2}{2} \pi_{\ell_{T_{t-1}(v)}}^{-} \right) \\
	& = \bar{\theta}^2 \int \frac{\left( \pi_{\ell_{T_{t-1}(v)}}^{+} - \pi_{\ell_{T_{t-1}(v)}}^{-} \right)^2}{ \frac{1-\theta_2}{2} \pi_{\ell_{T_{t-1}(v)}}^{+} + \frac{1+\theta_2}{2} \pi_{\ell_{T_{t-1}(v)}}^{-} } dx \\
	& \leq \bar{\theta}^2 \int \left( \pi_{\ell_{T_{t-1}(v)}}^{+} - \pi_{\ell_{T_{t-1}(v)}}^{-} \right)^2 \left[\frac{1-\theta_2}{2} \frac{1}{\pi_{\ell_{T_{t-1}(v)}}^{+}} + \frac{1+\theta_2}{2} \frac{1}{\pi_{\ell_{T_{t-1}(v)}}^{-}}  \right] dx \\
	& \leq \bar{\theta}^2 D_{T_{t-1}(v)},
\end{align*}
where the second to last step follows from Jensen's inequality for function $1/x$. Now we have the following recursion relationship
\begin{align*}
\log (1+D_{T_{t}(o)} ) \leq  \sum_{v \in \C(o)} \log (1+\bar{\theta}^2 \cdot D_{T_{t-1}(v)}).
\end{align*}
Invoke the following fact,
\begin{align*}
\frac{\log (1+\theta^2 x)}{\theta^2} \leq (1+\eta) \log(1+x) \quad \text{for all}~~0 \leq x \leq \eta,~\forall \theta,
\end{align*}
whose proof is in one line
$$
\frac{\log (1+\theta^2 x)}{\theta^2} \leq x \leq (1+\eta) \frac{x}{1+x} \leq (1+\eta) \log(1+x).
$$
Thus if $D_{T_{t-1}(v)} \leq \eta, \forall v\in \C(o)$, then the following holds
\begin{align}
\label{eq:low.recusion}
\log (1+D_{T_{t}(o)}) \leq  (1+\eta)\bar{\theta}^2 \sum_{v \in \C^{\rm u}(\rho)} \log (1+ D_{T_{t-1}(v)}).
\end{align}
Denoting
$$d_{T_{t}(o)}   := \log (1+D_{T_{t}(o)} ),$$
Equation~\eqref{eq:low.recusion} becomes
$$d_{T_{t}(o)} \leq  (1+\eta)\bar{\theta}^2 \sum_{v \in \C^{\rm u}(\rho)} d_{T_{t-1}(v)}.$$

	We will again need the Lemma~\ref{lem:cutset} that describes the branching number through the cutset.
Fix any $\lambda$ such that $\bar{\theta}^{-2} > \lambda >{\rm br}[T(o)]$.
For any $\epsilon$ small, Lemma~\ref{lem:cutset} claims the existence of cutset $C_\epsilon$ such that Eq.~\eqref{eq:cutset.all} and \eqref{eq:cutset.partial} holds.
Let's prove through induction on $\max_{x \in C_\epsilon} |x| - |v|$ that for any $v$ such that $|v| \leq \max_{x \in C_\epsilon} |x|$, we have
\begin{align}
\label{eq:induction}
d_{T_{C_{\epsilon}}(v)} \leq \frac{\eta}{1+\eta} \sum_{x \in C_{\epsilon} \cap T(v)} \left( \frac{1}{\lambda} \right)^{|x|-|v|} \leq \frac{\eta}{1+\eta} .
\end{align}
Note for the start of induction $v \in C_{\epsilon}$,
$$d_{T_{C_{\epsilon}}(v)} = \log (1 + \frac{4\delta^2}{1 - \delta^2}) <  \frac{\eta}{1+\eta}.$$ Now precede with the induction, assume for $u$ such that $\max_{x \in C_\epsilon} |x| - |u| = t-1$ equation~\eqref{eq:induction} is satisfied, let's prove for $v: \max_{x \in C_\epsilon} |x| - |v| = t$.  Due to the fact for all $u \in \C(v)$, $d_{T_{ C_{\epsilon}}(u)}  \leq \frac{\eta}{1+\eta} \Rightarrow D_{T_{C_{\epsilon}}(u)}  \leq \eta$, we can recall the linearized recursion
\begin{align*}
d_{T_{C_{\epsilon}}(v)} &\leq  (1+\eta)\bar{\theta}^2 \sum_{u \in \C(v)} d_{T_{\leq C_{\epsilon}}(u)} \\
& \leq  (1+\eta)\bar{\theta}^2 \sum_{u \in \C(v)} \left[  \frac{\eta}{1+\eta} \sum_{x \in C_{\epsilon} \cap T(u)} \left( \frac{1}{\lambda} \right)^{|x|-|u|} \right] \\
& \leq \frac{\eta}{1+\eta} \cdot (1+\eta)\bar{\theta}^2 \lambda  \sum_{u \in \C(v)} \sum_{x \in C_{\epsilon} \cap T(u)} \left( \frac{1}{\lambda} \right)^{|x|-|u| + 1}  \\
& \leq \eta \bar{\theta}^2 \lambda  \sum_{u \in \C(v)} \sum_{x \in C_{\epsilon} \cap T(u)} \left( \frac{1}{\lambda} \right)^{|x|-|v|}  \\
& \leq \eta \bar{\theta}^2 \lambda  \sum_{x \in C_{\epsilon} \cap T(v)} \left( \frac{1}{\lambda} \right)^{|x|-|v|} \leq \frac{\eta}{1+\eta} \sum_{x \in C_{\epsilon} \cap T(v)} \left( \frac{1}{\lambda} \right)^{|x|-|v|},
\end{align*}
if $ \bar{\theta}^2 \lambda \leq \frac{1}{1+\eta}$.
So far we have proved for any $v$, such that $|v| \leq \max_{x \in C_\epsilon} |x|$
\begin{align*}
d_{T_{\leq C_{\epsilon}}(v)} \leq \frac{\eta}{1+\eta} \sum_{x \in C_{\epsilon} \cap T(v)} \left( \frac{1}{\lambda} \right)^{|x|-|v|} \leq \frac{\eta}{1+\eta}\\
\text{which implies}\quad D_{T_{\leq C_{\epsilon}}(v)} \leq \eta
\end{align*}
so that the linearized recursion~\eqref{eq:low.recusion} always holds.
Take $\epsilon \rightarrow 0, \lambda \rightarrow {\rm br}[T(o)]$. Define $t_\epsilon: = \min\{ |x|, x\in C_{\epsilon} \}$, it is also easy to see from equation~\eqref{eq:cutset.all} that
$$
\left( \frac{1}{\lambda} \right)^{t_\epsilon} \leq \sum_{x \in C_\epsilon} \left( \frac{1}{\lambda} \right)^{|x|} \leq \epsilon \Rightarrow t_{\epsilon} > \frac{\log(1/\epsilon)}{\log \lambda} \rightarrow \infty.
$$
Putting things together, under the condition
\begin{align*}
 \log \left(1+\frac{4\delta^2}{1-\delta^2}\right) \leq 1 - {\rm br}[T(o)] \bar{\theta}^2,
\end{align*}
we have
$$\lim_{t \rightarrow \infty} D_{T_t(o)} = \lim_{\epsilon \rightarrow 0} D_{T_{C_\epsilon}(o) }\leq \frac{\eta}{1+\eta} \cdot \lim_{\epsilon \rightarrow 0} \sum_{x \in C_{\epsilon} \cap T(o)} \left( \frac{1}{\lambda} \right)^{|x|} = 0.$$

\end{proof}

\begin{proof}[Proof of Proposition~\ref{lma:coupling.tree}]
  The proof is a standard exercise following the idea from Proposition 4.2 in \cite{mossel2012stochastic}. First, let's recall Bernstein inequality. Consider
  $X \sim \text{Binom}(n, p_0)$, then the following concentration inequality holds
  $$\mathbb{P} (X \geq n p_0 + t) \leq \exp(-\frac{t^2}{2(n p_0 + t/3 )}).$$
  Hence if we plug in $t =  \frac{2}{3}\log n + \sqrt{2 np_0 \log n}$, we know
  $$|\partial G_1| \stackrel{sto.}{\leq} X \leq np_0 + \frac{2}{3}\log n + \sqrt{2 np_0 \log n}\leq 2np_0 + 2\log n $$ with probability at least $1 - n^{-1}$.

  Now, through union bound, we can prove that
  \begin{align*}
    \mathbb{P} \left(\forall r \leq R, |\partial G_r| \leq (2np_0  + 2\log n)^r  \right)
    \geq 1 - C \cdot (2np_0  + 2\log n)^{R} n^{-1} \geq 1 - O(n^{-3/4}).
  \end{align*}
  And we know that on the same event,
  $$ |\partial G_r| \leq n^{1/4}, \forall r \leq R.$$

It is clear that bad events that $G_R$ is not a tree (with cycles) for each layer is bounded above by
$p_0^2 |\partial G_r| + p_0 |\partial G_r|^2$. Take a further union bound over all layers, we know this probability is bounded by
$O(n^{-1/8})$ provided $p_0 = o(n^{-5/8})$.

Now we need to recursively use the Poisson-Binomial coupling (to achieve Poisson-Multinomial coupling). The following
Lemma is taken from \cite{mossel2012stochastic} (Lemma 4.6).
\begin{lemma}
  If  $m, n$ are positive integers then
  $$\| \textrm{Binom}(m, \frac{c}{n}) - \textrm{Poisson}(c) \|_{TV} \leq O( \frac{c^2 m}{n^2} + c|\frac{m}{n}-1|) $$
\end{lemma}

Now we condition on all the good events up to layer $G_{r-1}$, which happens with
probability at least $1 - n^{-1/8} - n^{-3/4}$. We can couple the next layer for nodes in $\partial G_{r}$.
Take a node $v \in \partial G_{r}$ as an example. Assume it is of color $i$,
then the number of color $j$ nodes in his children follows
$\text{Binom}(|V_{>r}^{i}|, p_{ij})$. Comparing to the Poisson version $\text{Poisson}(n_i p_{ij})$, we know
with probability at least
$$1 - O(n_i p^2_{ij} + p_{ij} |V_{>r}^{i} - n_i|),$$ one can couple the Poisson and Binomial in the same probability space.
Note that $|V_{>r}^{i} - n_i| \leq |\partial G_r|$.
Repeat this recursively, and use the union bound, we can couple
$(G_R, \ell_{G_R}) = (T_R, \ell_{T_R})$
with probability at least $1 -  O(k \max_i(n_i) p_0^2 + k p_0 n^{1/4}) n^{1/4} \log n = 1 - o(1)$.

Therefore if $n p_0 = n^{o(1)}$ and $k \precsim \log n$, we have the bad event (when we cannot couple) happens with probability going to $0$ as $n \rightarrow \infty$. And if $p_0 = n^{o(1)}$, we can allow $R$ to grow to infinity at a slow rate as $ R \precsim \frac{\log n}{  \log [n^{o(1)} +  \log n] }.$

\end{proof}

\section*{Acknowledgements}
The authors want to thank Elchanan Mossel for many valuable discussions.

\bibliographystyle{abbrv} 
\bibliography{bibfile}

\end{document}